\documentclass[11pt,a4paper]{article}
\usepackage[T1]{fontenc}
\usepackage{textcomp}
\usepackage[a4paper]{geometry}
\usepackage{amssymb,latexsym,amsmath,amsfonts,amsthm}
\usepackage{graphicx}
\usepackage{epstopdf}
\usepackage{algorithm}
\usepackage{algorithmic}
\usepackage{dsfont}
\usepackage{mathrsfs}
\usepackage{epsfig}
\usepackage{booktabs}
\usepackage{mathtools}
\usepackage{comment,verbatim}
\usepackage{subfigure}
\usepackage{overpic}
\usepackage{hyperref}
\usepackage{bbm}
\usepackage[normalem]{ulem}
\usepackage{cite}
\usepackage{xcolor}

\newtheorem{theorem}{Theorem}[section]
\newtheorem{definition}[theorem]{Definition}
\newtheorem{lemma}[theorem]{Lemma}

\makeatletter
\newtheoremstyle{plainnsl}
{8pt}
{8pt}
{\upshape}
{}
{\bfseries}
{.}
{ }
{}
\makeatother

\theoremstyle{plainnsl}
\newtheorem{example}[theorem]{Example}
\newtheorem{remark}[theorem]{Remark}

\makeatletter
\@addtoreset{equation}{section}
\makeatother

\graphicspath{ {images/} }

\usepackage{color}

\begin{document}
	
	\title{On exponential convergence of Chebyshev polynomial approximation for multivariate analytic functions}
	
	\author{Xinyu Wang\footnotemark[1] ~ and ~ Haiyong Wang\footnotemark[1]~\footnotemark[2]}
	
	\renewcommand{\thefootnote}{\fnsymbol{footnote}}
	
	\footnotetext[1]{School of Mathematics and Statistics, Huazhong University of Science and Technology, Wuhan 430074, P. R. China. \texttt{Email:haiyongwang@hust.edu.cn}}
	
	\footnotetext[2]{Hubei Key Laboratory of Engineering Modeling and Scientific Computing, Huazhong University of Science and Technology, Wuhan 430074, P. R. China.}
	
	\maketitle
	
	\begin{abstract}
		This paper presents a new analysis of the Chebyshev projection for multivariate analytic functions, drawing on pluripotential theory. It is proved that in any downward closed convex polynomial space, the Chebyshev projection achieves the same exponential convergence rate as the best polynomial approximation.
		This result enables a precise quantification of the exponential convergence rate of the Chebyshev projection.
		The analysis is then extended to several related topics, including tensorized Chebyshev interpolation, tensor product Gauss--Legendre quadrature, Padua interpolation and cubature, and Chebyshev-Galerkin method, with the corresponding exponential convergence rate established in each case. Supporting numerical experiments are provided to validate the theoretical results.
	\end{abstract}
	
	
	\vspace{0.05in}
	

	\section{Introduction}
	Chebyshev polynomial approximations play a crucial role in many areas of scientific computing, including spectral methods for solving PDEs \cite{boyd2001chebyshev,canuto2007spectral,shen2011spectral}, the Chebfun software for numerical computing \cite{driscoll2014chebfun}, option pricing and finance modeling \cite{gass2018chebyshev,glau2020low}, deep neural networks \cite{tang2024chebnet}, etc. Consider a function $f(\mathbf{x})$ defined on the $d$-dimensional hypercube $[-1,1]^d$ for some $d\in\mathbb{N}$, and let $T_k(x) = \cos(k\arccos(x))$ denote the Chebyshev polynomial of the first kind of degree $k$. If $f$ satisfies the Dini--Lipschitz condition \cite[Theorem 4.1]{mason1980near}, then it has a uniformly and absolutely Chebyshev expansion of the form
	\begin{equation}\label{onepone}
		f(\mathbf{x}) = \sum_{\mathbf{k}\in\mathbb{N}_0^d}a_\mathbf{k}T_\mathbf{k}(\mathbf{x}) = \sum_{k_1=0}^{\infty}\cdots\sum_{k_d=0}^{\infty}a_{k_1,\dots,k_d} T_{k_1}(x_1)\cdots T_{k_d}(x_d),
	\end{equation}
	where $\mathbb{N}_0^d$ denotes the set of all $d$-tuples of nonnegative integers. Truncating this expansion yields the Chebyshev projection of $f$:
	\begin{equation}
		S_{\Lambda}(f)(\mathbf{x}) = \sum_{\mathbf{k}\in \Lambda}a_{\mathbf{k}}T_{\mathbf{k}}(\mathbf{x}), \notag
	\end{equation}
	where $\Lambda \subset \mathbb{N}_0^d$ is a finite multi-index set chosen suitably. In the univariate case with $\Lambda = \{0,\dots,n\}$, the exponential convergence of the Chebyshev projection for analytic functions was first established by Bernstein in 1912: if $f$ is analytic in a neighbourhood of $\overline{E(\rho)}$, where $E(\rho)$ denotes the open region bounded by the {\it Bernstein ellipse}
	\begin{equation}\label{oneptwo}
		\partial E(\rho) := \left\{
		z\in\mathbb{C}: z=\frac{u+u^{-1}}{2},\ |u|=\rho
		\right\}, \quad \rho > 1,
		\notag
	\end{equation}
	then
	$\max_{x \in [-1,1]} \left|f(x)-S_{\Lambda}(f)(x)\right| = \mathcal{O}(\rho^{-n})$ \cite[pp.~94-95]{bernstein1912ordre}.
	The parameter $\rho$ can be explicitly determined from the locations of the singularities of $f$. In the multivariate case, however, only a few discussions have been devoted to analyzing the exponential convergence of Chebyshev projection. Key contributions can be summarized as follows:
	\begin{itemize}
		\item[$\bullet$] Bochner and Martin in \cite[Chapter V]{bochner1948several} established a multivariate analogue of Bernstein's result, showing that if $f$ is analytic in a neighbourhood of $\overline{E(\rho_1)}\times \cdots \times \overline{E(\rho_d)}$, then its Chebyshev coefficients satisfy $a_{\mathbf{k}} = \mathcal{O}(\rho_1^{-k_1} \cdots \rho_d^{-k_d})$ for $\mathbf{k} = (k_1, \ldots, k_d) \in \mathbb{N}_0^d$. Although the exponential convergence behavior of Chebyshev projection can be easily seen, the quantification of the exponential convergence rate remains open.
		
		\item[$\bullet$]
		More recently, based on the observation that multivariate polynomials of fixed degree have anisotropic resolution power in the hypercube, Trefethen in \cite{trefethen2017cubature,trefethen2017multivariate} introduced the {\it Euclidean degree} for multivariate polynomials, and quantified explicitly the exponential convergence rates of Chebyshev projections with total, Euclidean and maximal degrees. Although the exponential convergence rates for Euclidean and maximal degrees are sharp for Runge functions, e.g., $f(\mathbf{x})=(x_1^2+ \cdots + x_d^2+h^2)^{-1}$ with $h > 0$, the sharpness for other types of analytic functions is still unclear (see \cite{bos2018bernstein,hecht2026multivariate}). Moreover, we will show that even for Runge functions, the exponential convergence rate for total degree is suboptimal (see Remark \ref{rk:Trefethen}).	
\end{itemize}

The aim of this paper is to analyze the exponential convergence rate of Chebyshev approximation for multivariate analytic functions from the perspective of pluripotential theory. In recent years, pluripotential theory associated to convex bodies has seen considerable development, particularly in connection with multivariate polynomial approximation (see, e.g.,
\cite{bos2018bernstein,magnusson2024siciak,magnusson2025bernstein}). Based on the analyticity assumption determined by the $P$-extremal function of the hypercube, we derive a new upper bound estimate for Chebyshev coefficients and show that the exponential convergence of the Chebyshev projection using a downward closed convex polynomial space, which includes commonly used index sets induced by the 1-, 2- and $\infty$-norms as special cases, can be quantified by the {\it maximal convergence number}. Our result implies that the Chebyshev projection using a downward closed convex polynomial space achieves the best possible exponential rate, since the maximal convergence number actually characterizes the exponential rate of best polynomial approximation for multivariate analytic functions.
We further extend our analysis to several related topics, including tensorized Chebyshev interpolation, tensor product Gauss--Legendre quadrature, Padua interpolation and cubature, and Chebyshev-Galerkin method, and establish their exponential convergence rates for multivariate analytic functions as well. Compared with existing convergence results for tensorized Chebyshev interpolation, our new result is sharper (see Remark \ref{contrast}). For Padua interpolation and cubature, to the best of our knowledge, their exponential convergence rates are established for the first time.

The paper is organized as follows. In section \ref{section_bernsteintheory}, we review some basic results on polynomial approximation for multivariate analytic functions. In section \ref{section_chebprojection}, we analyze the exponential convergence rate of Chebyshev projection, with supporting numerical experiments provided in section \ref{section_numerical}. Finally, section \ref{section_application} applies our analysis results to the related topics, including tensorized Chebyshev interpolation, tensor product Gauss--Legendre quadrature, Padua interpolation and cubature, and Chebyshev-Galerkin method.

\section{Bernstein--Walsh theory associated to convex bodies}\label{section_bernsteintheory}
In this section, we briefly introduce the Bernstein--Walsh theory associated to convex bodies. We refer to \cite{bos2018bernstein,klimek1991pluripotential,magnusson2025bernstein} for more details. Throughout the paper, we use the standard multi-index notation, such as $\mathbf{z}^\mathbf{k}=\prod_{i=1}^d z_i^{k_i}$ for $\mathbf{z}=(z_1,\ldots,z_d)$ and $\mathbf{k}=(k_1,\ldots,k_d)$, and $\boldsymbol{\alpha} \le \boldsymbol{\beta}$ for $\boldsymbol{\alpha}=(\alpha_1,\ldots,\alpha_d)$ and $\boldsymbol{\beta}=(\beta_1,\ldots,\beta_d)$ if $\alpha_i \le \beta_i$, $1\le i\le d$. We also use $\Vert \cdot \Vert_{\ell^q}$, $1 \le q \le \infty$, to denote the standard $q$-norm of vectors in $\mathbb{R}^d$, and $\Vert \cdot \Vert_{\infty}$ to denote the maximum norm over $[-1,1]^d$.

\subsection{The $\boldsymbol{P}$-extremal functions and Bernstein--Walsh theorem}\label{section2.1}
We begin with a review of the classical extremal functions. Let $\mbox{PSH}(\mathbb{C}^d)$ denote the set of plurisubharmonic functions on $\mathbb{C}^d$, and let $L(\mathbb{C}^d)$ denote the usual Lelong class defined by
\begin{equation}
	L(\mathbb{C}^d) := \left\{ u\in \mbox{PSH}(\mathbb{C}^d) : u(\mathbf{z}) -  \max_{1\le i\le d}\{\log|z_i|\} = \mathcal{O}(1) \textnormal{ as } |\mathbf{z}| \rightarrow \infty \right\}, \notag
\end{equation}
where $|\mathbf{z}|$ is the modulus of $\mathbf{z} = (z_1, \dots, z_d) \in \mathbb{C}^d$. For a compact set $K\subset\mathbb{C}^d$, the {\it extremal function} (or the {\it pluricomplex Green function}) of $K$ is given by
\begin{equation}
	V_{K}(\mathbf{z}) := \sup \{ u(\mathbf{z}):u\in L(\mathbb{C}^d), \, u \le 0 \textnormal{ on } K\}. \notag
\end{equation}
Denote by $V_K^{*}(\mathbf{z}):=\limsup_{\boldsymbol{\zeta}\rightarrow{\mathbf{z}}}V_K(\boldsymbol{\zeta})$ the {\it upper semicontinuous regularization of $V_K$}. A compact set $K$ is said to be {\it $L$-regular} if $V_{K}$ is continuous on $K$. Since $V_K$ is automatically lower semicontinuous, this condition implies that $V_K^* = V_K$. It is known that the extremal function can be characterized using the subclass of $L(\mathbb{C}^d)$ arising from polynomials. Specifically,
\begin{equation}
	V_{K}(\mathbf{z}) = \max \left\{ 0, \, \sup \left\{ \frac{\log|p(\mathbf{z})|}{\textnormal{deg}\,p}:\|p\|_{K}:=\sup_{\mathbf{z}\in K}|p(\mathbf{z})| \le 1, \, \textnormal{deg}\,p \ge 1 \right\} \right\}, \notag
\end{equation}
where $p(\mathbf{z})$ is a complex polynomial with total degree $\textnormal{deg}\,p$.

We now introduce the $P$-extremal function as a generalization of the classical extremal function. Let $\mathbb{R}^+=[0,\infty)$ and let $\Sigma=\{ \mathbf{x} = (x_1,\dots,x_d) \in(\mathbb{R}^+)^d: x_1 + \cdots +x_d \le 1 \}$ be the unit simplex. Consider a convex body (i.e., a compact convex set with nonempty interior) $P\subset(\mathbb{R}^+)^d$ with the property that $\Sigma\subset\alpha P$ for some $\alpha>0$. We define the {\it logarithmic indicator function} of $P$ on $\mathbb{C}^d$ by
\begin{equation}
	H_{P}(\mathbf{z}) := \sup_{J\in P\,} \log |\mathbf{z}^{J}| =  \sup_{(j_1,\dots,j_d)\in P} \left(j_1\log|z_1|+\cdots+j_d\log|z_d|\right), \notag
\end{equation}
and introduce the generalization of Lelong class as
\begin{equation}\label{Lp}
	L_P(\mathbb{C}^d) := \left\{ u\in \mbox{PSH}(\mathbb{C}^d) : u(\mathbf{z})-H_P(\mathbf{z}) = \mathcal{O}(1) \textnormal{ as } |\mathbf{z}|\!\rightarrow\!\infty \right\}. \notag
\end{equation}
The $P$-{\it extremal function} of a compact set $K \subset \mathbb{C}^d$ is then defined as
\begin{equation}
	V_{P,K}(\mathbf{z}) := \sup  \{ u(\mathbf{z}):u\in L_P(\mathbb{C}^d), \, u \le 0 \textnormal{ on } K\}. \notag
\end{equation}
Let $V_{P,K}^*(\mathbf{z}) = \limsup_{\boldsymbol{\zeta}\rightarrow{\mathbf{z}}} V_{P,K}(\boldsymbol{\zeta})$. As in the classical case, we say that a compact set $K$ is {\it $PL$-regular} if $V_{P,K}$ is continuous on $K$, i.e., $V_{P,K}^{*}=V_{P,K}$. Note that if $K$ is $L$-regular, then it is $PL$-regular for any convex body $P$ as above.
Furthermore, when $P=\Sigma$, the generalized Lelong class reduces to the usual Lelong class, i.e.,  $L_{\Sigma}(\mathbb{C}^d)=L(\mathbb{C}^d)$, and the $P$-extremal function recovers the classical extremal function, i.e., $V_{\Sigma,K} = V_{K}$.

Define the finite-dimensional multivariate polynomial space
\begin{equation}
	\mbox{Poly} (nP) := \left\{ p(\mathbf{z})=\sum _{\mathbf{k}\in nP \cap \mathbb{N}_0^d} c_{\mathbf{k}} \mathbf{z}^{ \mathbf{k}},\, c_{\mathbf{k}}\in \mathbb{C} \right\} , \quad n \in \mathbb{N}, \notag
\end{equation}
and
\begin{equation}
	D_n(f,P,K) := \inf \left\{
	\Vert f-p \Vert_{K}
	:p \in \textnormal{Poly}(nP) \right\}. \notag
\end{equation}
Note that $\mbox{Poly} (nP)$ is the space of polynomials of total degree at most $n$ when $P=\Sigma$. The Bernstein--Walsh theorem associated to convex body $P$ is stated as follows \cite[Theorem 3.1]{bos2018bernstein}.
\begin{theorem}\label{bw2}
	Let $K \subset \mathbb{C}^d$ be compact and $PL$-regular, and let $f(\mathbf{z})$ be continuous on $K$. Let $R>1$ and define
	\begin{equation}
		\Omega(P,K,R)=\{ \mathbf{z} \in \mathbb{C}^d :V_{P,K}(\mathbf{z}) < \log R \}. \notag
	\end{equation}
	Then
	\begin{equation}
		\limsup_{n \to \infty} {D_n(f,P,K)}^{1/n} \le R^{-1} \notag
	\end{equation}
	if and only if $f(\mathbf{z})$ is the restriction of a function analytic in $\Omega(P,K,R)$ to $K$.
\end{theorem}

Bernstein--Walsh theorem gives a qualitative result on the best polynomial approximation for analytic functions. Specifically, it shows that $D_n(f,P,K)$ will decay at the exponential rate $\mathcal{O}(R_{\mbox{\scriptsize max}}^{-n})$, where
\begin{equation}
	R_{\mbox{\scriptsize max}} := \sup \{ R:\Omega(P,K,R)\cap S(f) = \emptyset\} = \inf \{ \exp(V_{P,K}(\mathbf{z})):\mathbf{z}\in S(f) \} \notag
\end{equation}
and $S(f)$ denotes the set of singularities of $f$. The number $R_{\mbox{\scriptsize max}}$ is called the {\it maximal convergence number} in \cite{kraus2011bernstein}, and it quantifies the exponential convergence rate of best polynomial approximation. We refer to \cite{bagby1993bernstein} for the history and proofs of the classical Bernstein--Walsh theorem, and to \cite{bos2018bernstein} for the proof of the Bernstein--Walsh theorem associated to convex bodies.

\subsection{Downward closed convex bodies, nonnegative polars, and the $\boldsymbol{P}$-extremal function of the hypercube}
We begin with the definition of downward closed convex bodies and their nonnegative polars.
\begin{definition}
	A convex body $P \subset (\mathbb{R}^+)^d$ is downward closed if $\boldsymbol{\mu} \le \boldsymbol{\nu} \in P$ implies $\boldsymbol{\mu} \in P$ for all $\boldsymbol{\mu} \in (\mathbb{R}^+)^d$.
	The nonnegative polar of $P$ is defined by
	\begin{equation}
		P^{*} := \left\{ \mathbf{x} \in (\mathbb{R}^{+})^d :\sup _{ \mathbf{y} \in P} \left(\mathbf{x}\cdot \mathbf{y} \right) \le 1  \right\}. \notag
	\end{equation}
\end{definition}

\begin{remark}
	The {\it polar} of a subset $P\subset \mathbb{R}^d$ is defined by $P^{\circ} := \{ \mathbf{x} \in \mathbb{R}^d : \sup _{ \mathbf{y} \in P} (\mathbf{x}\cdot \mathbf{y} ) \le 1 \}$ (see \cite[p.~32]{schneider2014convex}). Clearly, the nonnegative polar $P^*$ defined above is the nonnegative portion of the polar $P^{\circ}$.
\end{remark}

Below we give two examples of downward closed convex bodies and their nonnegative polars, followed by an example of a convex body that is not downward closed.
\begin{itemize}
	\item[(i)] The nonnegative portion of the $\ell^{q}$ unit ball $\mathbb{B}_{q}^{+}=\{ \mathbf{x} \in (\mathbb{R}^+)^d: \Vert \mathbf{x} \Vert_{\ell^q} \le 1 \}$ for $1\le q\le \infty$ is a downward closed convex body. Moreover, the nonnegative polar of $\mathbb{B}_{q}^{+}$ is $\mathbb{B}_{p}^{+}$, where $p$ is the conjugate exponent of $q$, i.e., $1/q + 1/p = 1$.
	\item[(ii)] Let $L = [0,r_1]\times \cdots \times [0,r_d]$ and $T = \big\{ \! \sum_{i=1}^{d} \lambda_i \mathbf{e}_i/r_i: \sum_{i=1}^{d} \lambda_i \le 1, \, \lambda_i \ge 0 \big\}$, where $r_i>0$ for $i=1,\ldots,d$, and $\{\mathbf{e}_i\}_{i=1}^d$ is the standard basis of $\mathbb{R}^d$. Both $L$ and $T$ are downward closed convex bodies, and $L$ is the nonnegative polar of $T$ (and vice versa).
	\item[(iii)] Let $G = \operatorname{conv} \left\{ \mathbf{0}, \mathbf{e}_1, \ldots, \mathbf{e}_d, 2\mathbf{e}_1+2\mathbf{e}_2 \right\}$, where $d \ge 2$ and $\operatorname{conv}$ denotes the convex hull. Then $G$ is a convex body satisfying $\Sigma\subset G$, but it is not downward closed. Indeed, $2\mathbf{e}_1 \le 2\mathbf{e}_1+2\mathbf{e}_2 \in G$, whereas $2\mathbf{e}_1\notin G$.
\end{itemize}

For a downward closed convex body $P \subset (\mathbb{R}^+)^d$, we define its centrally symmetric extension to $\mathbb{R}^d$ by
\begin{equation}
	C(P) := \left\{ \mathbf{x} \in \mathbb{R}^d: \mathbf{x}^+ := (|x_1|, \dots, |x_d|) \in P \right\}. \notag
\end{equation}
The set $C(P)$ is convex. Indeed, if
$\boldsymbol{\mu},\boldsymbol{\nu}\in C(P)$, then by the triangle inequality and the convexity of $P$, we have
\[
\left(\theta\boldsymbol{\mu}+(1-\theta)\boldsymbol{\nu}\right)^+
\le
\theta\boldsymbol{\mu}^+ +(1-\theta)\boldsymbol{\nu}^+ \in P, \quad 0 < \theta < 1,
\]
which, by the downward closed property of $P$, implies
$\theta\boldsymbol{\mu}+(1-\theta)\boldsymbol{\nu}\in C(P)$. The {\it gauge function} (or the {\it Minkowski functional}) of $C(P)$ is defined as
\begin{equation}\label{norm}
	\Vert \mathbf{x} \Vert_{P} :=
	\inf \left\{ \lambda\geq0: \mathbf{x} \in \lambda C(P) \right\}, \quad \mathbf{x}\in\mathbb{R}^{d}.
\end{equation}
Since $C(P)$ is a convex body that is centrally symmetric with respect to the origin, the associated gauge function indeed defines a norm on $\mathbb{R}^d$.
Moreover, it is easily verified that $P^*$ is also a downward closed convex body. Consequently, we may define the gauge function $\Vert \cdot \Vert_{P^*}$ in a manner analogous to \eqref{norm}.
The following lemma will be useful.
\begin{lemma}\label{lemma1}
	Let $P \subset (\mathbb{R}^+)^d$ be a downward closed convex body. Then the following statements are true.
	\begin{itemize}
		\setlength{\itemsep}{0pt}
		\item[\rm (i)]
		For any $\mathbf{x} \in \mathbb{R}^d$, it holds that
		$\Vert \mathbf{x} \Vert_{P^*} = \sup _{\mathbf{y} \in C(P)} (\mathbf{x}\cdot\mathbf{y})$.
		
		\item[\rm (ii)]
		For any $\mathbf{x},\mathbf{y} \in \mathbb{R}^d$, it holds that
		$|\mathbf{x} \cdot \mathbf{y}| \le \Vert \mathbf{x} \Vert_{P} \Vert \mathbf{y} \Vert_{P^*}$. Moreover, for each $\mathbf{x} \in \mathbb{R}^d$, there exists $\mathbf{y} \in \mathbb{R}^d$ satisfying $\Vert \mathbf{y} \Vert_{P^*} = 1$ such that $|\mathbf{x} \cdot \mathbf{y}|=\Vert \mathbf{x} \Vert_{P} $.
	\end{itemize}
\end{lemma}
\begin{proof}
	For (i), since $C(P^*)$ is a convex body containing the origin, it follows from \cite[Lemma 1.7.13]{schneider2014convex} that
	\begin{equation}
		\Vert \mathbf{x} \Vert_{P^*} = \sup _{\mathbf{y} \in (C(P^*))^{\circ}} (\mathbf{x}\cdot\mathbf{y}), \notag
	\end{equation}
	where $(\cdot)^{\circ}$ denotes the polar. Moreover, since $C(P^*) = (C(P))^{\circ}$ and $C(P)$ is also a convex body containing the origin, it follows from \cite[Theorem 1.6.2]{schneider2014convex} that $(C(P^*))^{\circ} = (C(P))^{\circ\circ} = C(P)$. Substituting this into the above identity yields the desired result.
	
	For (ii), since $C(P)=\{x\in\mathbb{R}^d:\|x\|_P\leq1\}$, statement (i) implies that $\Vert \cdot \Vert_{P^*}$ is the dual norm of $\Vert \cdot \Vert_{P}$. The second statement then follows from \cite[Appendix A.1.6]{boyd2004convex}.	
\end{proof}

The Bernstein--Walsh theorem holds for a general compact set $K \subset \mathbb{C}^d$, yet the closed form of the $P$-extremal function $V_{P,K}$ remains elusive in most cases. However, when $P$ is a downward closed convex body and $K$ is a Cartesian product of the form $K = K_1 \times \cdots \times K_d$, where each $K_i \subset \mathbb{C}$ is compact and non-polar
(i.e., $K_i$ has a positive logarithmic capacity),
it has been shown in \cite[Proposition~2.4]{bos2018bernstein} that the $P$-extremal function of $K$ can be expressed as
\begin{equation}
	V_{P,K}^{*}(z_1, \dots, z_d) = \sup_{(y_1,\ldots,y_d)\in P} \left( y_1 V_{K_1}^{*}(z_1) + \cdots + y_d V_{K_d}^{*}(z_d) \right). \notag
\end{equation}
In the case of the hypercube, i.e., $K_i=[-1,1]$ for $i=1,\ldots,d$, the extremal function for each $K_i$ is given by $V_{K_i}(z)=\log\rho(z)$, in which $\rho(z) = |z + \sqrt{z^2 - 1}|$ and the branch is chosen so that $\rho(z) \ge 1$. Consequently, the $P$-extremal function of the hypercube $K = [-1,1]^d$ has the closed form
\begin{equation}\label{p[-1,1]}
	V_{P,K}(\mathbf{z}) = \sup_{\mathbf{y} \in P} \big(\log {\rho}(\mathbf{z}) \cdot \mathbf{y} \big)
	= \left\Vert \log {\rho}(\mathbf{z}) \right\Vert_{P^*}, \notag
\end{equation}
where $\log {\rho}(\mathbf{z}) = \left( \log \rho(z_1), \dots, \log \rho(z_d) \right)$, and the second equality follows from the item (i) of Lemma \ref{lemma1}.

\section{Exponential convergence analysis of multivariate Chebyshev projection}\label{section_chebprojection}
In this section, we present a convergence analysis for multivariate Chebyshev projection. We restrict our attention to index sets of the form $nP \cap \mathbb{N}_0^d$, where $n \in \mathbb{N}$ and $P$ is a downward closed convex body.	For $R>1$, we introduce the analyticity domain
\begin{equation}\label{def:AnalRegion}
	\Omega(P,R):= \{ \mathbf{z} \in \mathbb{C}^d :V_{P,[-1,1]^d}(\mathbf{z}) = \Vert \log \rho (\mathbf{z}) \Vert_{P^*} < \log R \}.  \notag
\end{equation}
Note that when $d = 1$ and $P = [0,1]$, $\Omega(P,R)$ reduces to $E(R)$.
In comparison with the two existing analyticity assumptions, namely the Bernstein polyellipse introduced in \cite[Chapter V]{bochner1948several} and the Newton ellipse introduced in \cite{trefethen2017multivariate}, our analyticity assumption can lead to the optimal exponential convergence rate.

\subsection{A new upper bound of Chebyshev coefficients}
Recall that the Chebyshev coefficient $a_{\mathbf{k}}$ in (\ref{onepone}) is defined as
\begin{equation}	
	a_{\mathbf{k}}=\frac{2^{d-\aleph(\mathbf{k})}}{\pi^d} \int_{[-1,1]^d}f(\mathbf{x})T_{\mathbf{k}}(\mathbf{x})\omega(\mathbf{x})\mathrm{d}\mathbf{x}, \notag
\end{equation}
where $\omega(\mathbf{x})=\prod_{i=1}^d (1-x_i^2)^{-1/2}$ and $\aleph(\mathbf{k})=\#\{ i:k_i\!=\!0 \}$. The following estimate of Chebyshev coefficients will be used to establish our results, which can be found in \cite[p. ~95]{bochner1948several} and \cite[Corollary 3.2]{wang2020analysis}. Here we adopt the latter statement.
\begin{lemma}\label{lemmaB}
Assume that $f(\mathbf{z})$ is analytic in a neighbourhood of $E = \overline{E(\rho_1)}\times \cdots \times \overline{E(\rho_d)}$. Then for each $\mathbf{k} =(k_1,\dots,k_d)\in \mathbb{N}_0^d$, it holds that
	\begin{equation}
		|a_\mathbf{k}|\le 2^{d-\aleph(\mathbf{k})} \,\max_{\mathbf{z}\in E}|f(\mathbf{z})|\, \rho_1^{-k_1} \cdots
		\rho_d^{-k_d}. \notag
	\end{equation}
\end{lemma}
The main result of this subsection is stated as follows.
\begin{theorem}\label{firstthm}
Assume that $f(\mathbf{z})$ is analytic in a neighbourhood of $\overline{\Omega(P,R)}$.
	Then for each $\mathbf{k} \in \mathbb{N}_0^d$, it holds that
	\begin{equation}
		|a_\mathbf{k}|\le
		2^{d-\aleph(\mathbf{k})} \!
		\max_{\!\mathbf{z}\in \overline{\Omega(P,R)}
		}|f(\mathbf{z})|\, R^{-\Vert \mathbf{k} \Vert_{P}}. \notag
	\end{equation}
\end{theorem}
\begin{proof}
	Let $\boldsymbol{\alpha}=(\alpha_1,\dots,\alpha_d)$ be an arbitrary nonnegative vector satisfying $\|\boldsymbol{\alpha}\|_{P^*}=1$.
	We claim that
	\begin{equation}
		\overline{E(R^{\, \alpha_1})} \times \cdots \times \overline{E(R^{\, \alpha_d})} \subset \overline{ \Omega(P,R)}. \notag
	\end{equation}
	Indeed, if $\mathbf{z} = (z_1,\dots,z_d)$ satisfies $\rho (z_j) \le R^{\, \alpha_j}$ for $1\le j\le d$, then we have
	\begin{equation}
		\Vert \log\rho (\mathbf{z})\Vert_{P^*}
		= \sup_{ \mathbf{y} \in P} \big( \log \rho (\mathbf{z}) \cdot \mathbf{y} \big) \le \left(\sup_{ \mathbf{y} \in P} \, \boldsymbol{\alpha} \cdot \mathbf{y} \right) \log R = \log R , \notag
	\end{equation}
	where we have used the item (i) of Lemma \ref{lemma1}. By Lemma \ref{lemmaB} with $ E = \overline{E(R^{\, \alpha_1})} \times \cdots \times
	\overline{E(R^{\, \alpha_d})} $, we obtain
	\begin{equation} \label{guess1}
		|a_\mathbf{k}| \le 2^{d-\aleph(\mathbf{k})} \, \max_{\mathbf{z}\in E}|f(\mathbf{z})| \,
		R^{ -\boldsymbol{\alpha} \cdot \mathbf{k} } \le
		2^{d-\aleph(\mathbf{k})} \!
		\max_{\!\mathbf{z}\in \overline{\Omega(P,R)}}|f(\mathbf{z})| \,R^{-\boldsymbol{\alpha} \cdot \mathbf{k}}. \notag
	\end{equation}
	From the item (ii) of Lemma \ref{lemma1}, we know that for any $\mathbf{k} \in \mathbb{N}_0^d$, there exists $\boldsymbol{\alpha}_\mathbf{k} \in (\mathbb{R}^+)^d$ satisfying $\Vert \boldsymbol{\alpha}_\mathbf{k} \Vert_{P^*}=1$ such that $\boldsymbol{\alpha}_\mathbf{k} \cdot \mathbf{k}=\Vert \mathbf{k} \Vert_{P}$. Note that $\boldsymbol{\alpha}$ is arbitrary in the above inequality,
	so we can set $\boldsymbol{\alpha} = \boldsymbol{\alpha}_\mathbf{k}$ and obtain
	\begin{equation}
		|a_\mathbf{k}| \le 2^{d-\aleph(\mathbf{k})}  \!
		\max_{\!\mathbf{z}\in \overline{\Omega(P,R)}}|f(\mathbf{z})|\, R^{-\boldsymbol{\alpha}_\mathbf{k} \cdot \mathbf{k}}
		= 2^{d-\aleph(\mathbf{k})} \!
		\max_{\!\mathbf{z}\in \overline{\Omega(P,R)}}|f(\mathbf{z})|\, R^{-\Vert \mathbf{k} \Vert_{P}}. \notag
	\end{equation}
	This completes the proof.
\end{proof}

\subsection{Exponential convergence rate of multivariate Chebyshev projection}
The following lemma will be used in the convergence analysis for Chebyshev projection.
\begin{lemma}\label{lemmabig}
	Let $\alpha\ge 0$ and $R>1$. For $\gamma>0$, it holds that
	\begin{equation}\label{bigineq}
		\int_{\mathbf{x} \ge \mathbf{0}, \Vert \mathbf{x} \Vert_{P} \ge \gamma}  \Vert \mathbf{x} \Vert_{P}^{\alpha}
		R^{ -\Vert \mathbf{x} \Vert_{P}}  \mathrm{d}\mathbf{x} \le C_1(d,P)
		\left[\sum_{k=1}^{\lceil \alpha+d \rceil} \frac{\gamma^{\alpha+d-k} }{(\log R)^k}\prod_{j=2}^{k}(\alpha + d +  1 - j)\right] \!R^{-\gamma}, \notag
	\end{equation}
	where $\lceil \cdot \rceil$ is the ceiling function, and the product on the right-hand side is assumed to be one when $k=1$. The constant $C_1(d,P)$ is defined by
	\begin{equation}\label{const1}
	C_1(d,P)=	
			\begin{cases}
				\max_{x \in P} x, &d = 1, \\[1ex]
				(\min_{\Vert \mathbf{x}\Vert_{\ell^2}=1} \Vert \mathbf{x}\Vert_{P})^{-d}
				(\pi/2)^{d-\lceil d/2 \rceil} /(d-2)!!, &d \ge 2.
		\end{cases}
	\end{equation}
\end{lemma}
\begin{proof}
We denote by $\Psi_{\mathbf{x}}$ the integral on the left-hand side. Using the spherical coordinate transformation $\mathbf{x}=rh(\boldsymbol{\varphi})$ with $h(\boldsymbol{\varphi}) = \left(\cos\varphi_1,\, \dots, \,\sin\varphi_1\sin\varphi_2\cdots\sin\varphi_{d-2}\sin\varphi_{d-1} \right)$, $\boldsymbol{\varphi} \in (0,\pi/2)^{d-1}$,
	we obtain
	\begin{equation}
			\Psi_{\mathbf{x}}
			= \int_{0}^{\pi/2} \cdots \int_{0}^{\pi/2} \left(\int_{\gamma/\Vert h(\boldsymbol{\varphi}) \Vert_{P}} ^{+\infty} r^{\alpha+d-1} R^{-r \Vert h(\boldsymbol{\varphi}) \Vert_{P}}\mathrm{d}r\right) \Vert h(\boldsymbol{\varphi}) \Vert_{P}^{\alpha} I(\boldsymbol{\varphi})\mathrm{d}\boldsymbol{\varphi} \notag
	\end{equation}
	where $I(\boldsymbol{\varphi}) = \sin^{d-2}\varphi_1 \sin^{d-3}\varphi_2 \cdots \sin\varphi_{d-2}$.
	Let $\Psi_{r}$ denote the integral in parentheses. From \cite[Lemma 4.3]{wang2020analysis} we know that
	\begin{equation}
		\Psi_{r} \le \left[ \sum_{k=1}^{\lceil \alpha+d \rceil} \frac{\gamma^{\alpha+d-k}}{(\log R)^k \Vert h(\boldsymbol{\varphi}) \Vert_{P}^{\alpha+d}} \prod_{j=2}^{k}(\alpha + d + 1 - j) \right]  R^{-\gamma}. \notag
	\end{equation}
	Substituting this into $\Psi_{\mathbf{x}}$ gives
	\begin{equation}\label{r2}
		\Psi_{\mathbf{x}} \le \left(\int_{0}^{\pi/2} \cdots \int_{0}^{\pi/2} \frac{I(\boldsymbol{\varphi}) }{\Vert h(\boldsymbol{\varphi}) \Vert_{P}^{d}}
		\mathrm{d}\boldsymbol{\varphi} \right) \left[\sum_{k=1}^{\lceil \alpha+d \rceil} \frac{\gamma^{\alpha+d-k} }{(\log R)^k}\prod_{j=2}^{k}(\alpha +  d  + 1 - j)\right] R^{-\gamma}. \notag
	\end{equation}
	We estimate the integral over $\boldsymbol{\varphi}$ in the above inequality by noting that $\Vert h(\boldsymbol{\varphi}) \Vert_{\ell^2} \equiv 1$ and applying the inequality $\Vert h(\boldsymbol{\varphi}) \Vert_{P}\ge c \Vert h(\boldsymbol{\varphi}) \Vert_{\ell^2}$, with $c = \min_{\Vert \mathbf{x}\Vert_{\ell^2}=1} \Vert \mathbf{x}\Vert_{P} > 0$. Substituting these into the above inequality yields the desired result.
\end{proof}

We now state the main result of this work. Let $S_{nP}(f)$ denote the multivariate Chebyshev projection of $f$ with the index set $nP\cap \mathbb{N}_0^d$, i.e.,
\begin{equation}
	S_{nP}(f)(\mathbf{x})=\sum_{\Vert \mathbf{k} \Vert_{P} \le n}a_{\mathbf{k}}T_{\mathbf{k}}(\mathbf{x}). \notag
\end{equation}
In the following, we establish a quantitative result on the exponential convergence rate of $S_{nP}(f)$ in the maximum norm over the hypercube $[-1,1]^d$. The notation $\mathbf{1}$ refers to $(1,\dots,1)$ in the rest of the paper.
\begin{theorem}\label{thm3}
Assume that \(f(\mathbf{z})\) is analytic in a neighbourhood of \(\overline{\Omega(P,R)}\). Then for \(n\in\mathbb{N}\), it holds that
	\begin{equation}\label{imp2}
		\Vert f-S_{nP}(f) \Vert_{\infty}
		\le
		C_2(d,P,R,f)
		\left[
		\sum_{k=1}^{d}
		\frac{(d-1)!}{(d-k)!}
		\frac{n^{d-k}}{(\log R)^k}
		\right]
		R^{-n},
\end{equation}
where
\begin{equation}\label{const2}
C_2(d,P,R,f) =	2^d R^{\Vert \mathbf{1}\Vert_P} C_1(d,P) \max_{\mathbf{z}\in \overline{\Omega(P,R)}} |f(\mathbf{z})|,
\end{equation}
and \(C_1(d,P)\) is defined in \textnormal{\eqref{const1}}.
\end{theorem}
\begin{proof}
Using the fact that $\Vert T_{\mathbf{k}} \Vert_{\infty} \le 1$ for all $\mathbf{k} \in \mathbb{N}_0^d$ and Theorem \ref{firstthm}, we have
\begin{equation} \label{ineq1}
		\Vert f-S_{nP}(f) \Vert_{\infty} \le
		\sum_{\Vert \mathbf{k} \Vert_{P} > n} |a_\mathbf{k}| \le 2^d \max _{\mathbf{z}\in \overline{\Omega(P,R)}
		}|f(\mathbf{z})|
		\sum_{\Vert \mathbf{k} \Vert_{P} > n} R^{-\Vert \mathbf{k} \Vert_{P} }.
\end{equation}
To estimate the last sum in \eqref{ineq1}, for each $\mathbf{k}=(k_1,\dots,k_d)\in\mathbb{N}_0^d$, we introduce the unit hypercube
	\begin{equation}
		H_{\mathbf{k}} = [k_1,k_1+1]\times\cdots\times[k_d,k_d+1]. \notag
	\end{equation}
	Since $\mathbf{x}\in H_{\mathbf{k}}$ implies $\Vert\mathbf{x}\Vert_P \le \Vert\mathbf{k}\Vert_P+\Vert\mathbf{1}\Vert_P$, we have
	\begin{equation}
		R^{-\Vert\mathbf{k}\Vert_P} \le R^{\Vert\mathbf{1}\Vert_P} R^{-\Vert\mathbf{x}\Vert_P}, \quad \mathbf{x}\in H_{\mathbf{k}}. \notag
	\end{equation}
	Using the above inequality, we obtain
	\begin{equation}\label{ineq2}
			\sum_{\Vert \mathbf{k}\Vert_P>n}
			R^{-\Vert \mathbf{k}\Vert_P}
			=
			\sum_{\Vert \mathbf{k}\Vert_P>n}
			\int_{H_{\mathbf{k}}}
			R^{-\Vert \mathbf{k}\Vert_P}\mathrm{d}\mathbf{x}
			\le
			R^{\Vert\mathbf{1}\Vert_P}
			\int_{\bigcup_{\Vert \mathbf{k}\Vert_P>n}H_{\mathbf{k}}}
			R^{-\Vert \mathbf{x}\Vert_P}\mathrm{d}\mathbf{x}.
	\end{equation}
	Moreover, since \(\mathbf{x}\in H_{\mathbf{k}}\) implies
	\(\Vert \mathbf{x}\Vert_P\ge \Vert \mathbf{k}\Vert_P\), we have
	\[
	\bigcup_{\Vert \mathbf{k}\Vert_P>n}H_{\mathbf{k}}
	\subset
	\left\{
	\mathbf{x}\in(\mathbb{R}^+)^d:
	\Vert \mathbf{x}\Vert_P>n
	\right\}.
	\]
	Using the above relation and Lemma \ref{lemmabig} with \(\alpha=0\) and \(\gamma=n\), we obtain
	\begin{equation}
		\begin{split}
			\int_{\bigcup_{\Vert \mathbf{k}\Vert_P>n}H_{\mathbf{k}}}
			R^{-\Vert \mathbf{x}\Vert_P}\mathrm{d}\mathbf{x}
			&\le
			\int_{\mathbf{x}\ge\mathbf{0}, \Vert \mathbf{x}\Vert_P>n}
			R^{-\Vert \mathbf{x}\Vert_P}\mathrm{d}\mathbf{x} \\
			&\le
			C_1(d,P)
			\left[
			\sum_{k=1}^{d}
			\frac{(d-1)!}{(d-k)!}
			\frac{n^{d-k}}{(\log R)^k}
			\right]
			R^{-n}. \notag
		\end{split}
	\end{equation}
	Combining \eqref{ineq1}, \eqref{ineq2} and the above inequality gives \eqref{imp2}. This completes the proof.
\end{proof}

Combining Theorem \ref{thm3} with the Bernstein--Walsh theorem, we conclude that when the index set is induced by a downward closed convex body $P$, the maximum errors of both multivariate Chebyshev projection and best polynomial approximation converge at the same exponential rate $\mathcal{O}(R_P^{-n})$. Here, the maximal convergence number $R_P$ is given by
\begin{equation}\label{RP}
	R_P := \inf _{\mathbf{z} \in S(f)} \exp(\Vert \log \rho (\mathbf{z}) \Vert_{P^*}),
\end{equation}
where $S(f)$ denotes the set of singularities of $f$. Note that the comparison of the convergence rates of Chebyshev, Legendre and their best counterparts in the univariate case (i.e., $d=1$) has been comprehensively discussed in \cite{wang2021much}.

In the special cases where the index sets are induced by 1-, 2-, and $\infty$-norm, i.e., $P=\Sigma$, $P=\mathbb{B}_{2}^{+}$ and $P=\Sigma^*$, which correspond to total, Euclidean and maximal degrees respectively, the maximal convergence numbers are respectively given by
\begin{align}
	& R_{\Sigma} = \inf _{ \mathbf{z} \in S(f)}  \exp\left(\Vert \log \rho (\mathbf{z}) \Vert_{\ell^{\infty}}\right) = \inf _{ \mathbf{z} \in S(f)}  \max_{1\leq i\leq d} \{ \rho(z_i) \},  \notag \\
	& R_{\mathbb{B}_{2}^{+}} = \inf _{ \mathbf{z} \in S(f)}  \exp\left(\Vert \log \rho (\mathbf{z}) \Vert_{\ell^{2}}\right) = \inf _{ \mathbf{z} \in S(f)} \exp\left[ \bigg(\sum_{i=1}^{d} \left( \log\rho(z_i) \right)^2\bigg)^{1/2} \right] ,  \notag \\
	& R_{\Sigma^*} = \inf _{ \mathbf{z} \in S(f)}  \exp\left(\Vert \log \rho (\mathbf{z}) \Vert_{\ell^1}\right) = \inf _{ \mathbf{z} \in S(f)}  \prod _{i=1}^{d}\rho(z_i).  \notag
\end{align}

\begin{remark}
	The maximal convergence number $R_P$ in \eqref{RP} is an optimization problem in multiple complex variables, where the objective function $\exp(\Vert \log \rho (\mathbf{z}) \Vert_{P^*})$ is nonlinear and non-analytic, and the constraint $\mathbf{z} \in S(f)$ is usually nonlinear as well. Except for some specific functions, it is difficult to derive a closed form for $R_P$. Nevertheless, certain optimization algorithms based on Wirtinger calculus may provide an efficient approach for computing $R_P$ \cite{sorber2012unconstrained}. We leave this problem for future research.
\end{remark}

\section{Numerical experiments}\label{section_numerical}
From Theorem \ref{thm3} we can deduce that the maximum error of Chebyshev projection of total degree $n$, i.e., $S_{n\Sigma}(f)$, decays like $\mathcal{O}(R_\Sigma^{-n})$, and the maximum error of Chebyshev projection of maximal degree $n$, i.e., $S_{n\Sigma^*}(f)$, decays like $\mathcal{O}(R_{\Sigma^{*}}^{-n})$. In this section, we present several numerical experiments to verify our theoretical results. The maximum errors are estimated as the largest pointwise errors on a uniform tensor-product grid with \(1001\) points in each coordinate direction over \([-1,1]^d\).

\begin{example}\label{example1}
	We consider the function
	\begin{equation}\label{func1}
		f(\mathbf{x}) = \frac{1}{\prod_{i=1}^{d}x_i^{r_i} + c} \notag
	\end{equation}
	with $\{ r_i\}_{i=1}^d \subset \mathbb{N}$, $c\in \mathbb{R}$, and
	discuss three cases: $\mathcal{A}_1=\{r_i$ is odd for some $i$, $|c|>1\}$, $\mathcal{A}_2=\{r_i$ is even for any $i$, $c<-1\}$ and $\mathcal{A}_3=\{r_i$ is even for any $i$, $c>0\}$.
	As shown in Appendix \ref{appendixA}, for the first two cases, we have
	\begin{equation} \label{Rone}
		R_\Sigma = \rho\left(|c|^{1/ \Vert \mathbf{r} \Vert_{\ell^1}}\right), \qquad  R_{\Sigma^{*}} = \rho\left(|c|^{1/ \Vert \mathbf{r} \Vert_{\ell^{\infty}}}\right),
	\end{equation}
	where $\mathbf{r}=(r_1,\dots,r_d)$.
	For the last case, we do not know the closed forms of $R_\Sigma$ and $R_{\Sigma^{*}}$ for the time being.
However, for the function $f(x_1,x_2) = (x_1^2 x_2^2 + c)^{-1}$ with $c>0$, which is a special example of $\mathcal{A}_3$ (i.e., $d=2$, $r_1 = r_2 = 2$), as shown in Appendix \ref{appendixA}, we have
	\begin{equation} \label{Rspec}
		R_\Sigma=\rho\left(\sqrt{s+1}\right) \textnormal{ with } s=\frac{\sqrt{1+4c}-1}{2}, \qquad R_{\Sigma^{*}}=\rho\left(\mathrm{i}\sqrt{c}\right).
	\end{equation}

Figure \ref{projection1} illustrates the maximum errors of $S_{n\Sigma}(f)$ and $S_{n\Sigma^*}(f)$ for
\begin{equation}
	f(x_1,x_2) = \frac{1}{x_1^2x_2^3+1.3}, \qquad \frac{1}{x_1^2x_2^2-1.1}, \qquad \frac{1}{x_1^2x_2^2+0.4}, \notag
\end{equation}
which belong to $\mathcal{A}_1$, $\mathcal{A}_2$ and $\mathcal{A}_3$ respectively. We see that the numerical results are consistent with our theoretical results.

\begin{figure}
	\centering
	\includegraphics[width=7.2cm,height=5.5cm]{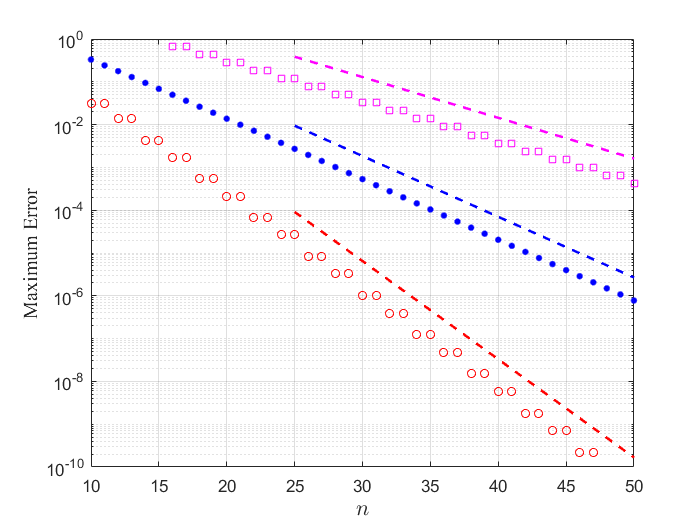}
	\includegraphics[width=7.2cm,height=5.5cm]{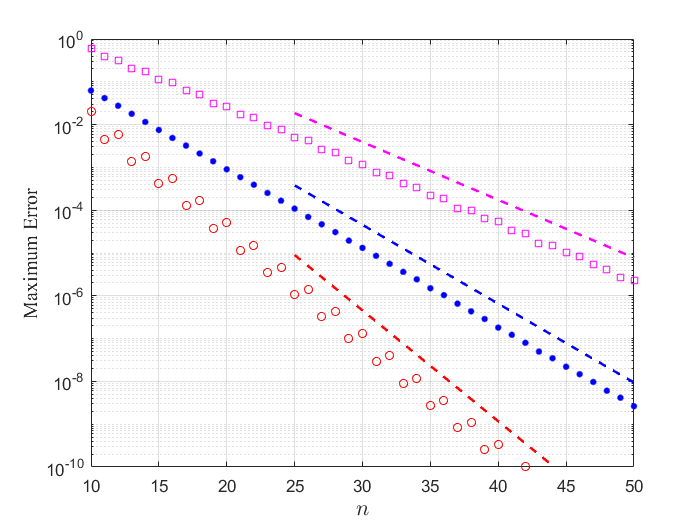}
	\caption{The maximum errors $\Vert f-S_{n\Sigma}(f) \Vert_{\infty}$ (left) and $\Vert f-S_{n\Sigma^*}(f) \Vert_{\infty}$ (right) for $f(x_1,x_2)=(x_1^2x_2^3+1.3)^{-1}$ (dots), $(x_1^2x_2^2-1.1)^{-1}$ (boxes) and $(x_1^2x_2^2+0.4)^{-1}$ (circles). The dashed lines denote the predicted convergence rates $\mathcal{O}(R_\Sigma^{-n})$ (left) and $\mathcal{O}(R_{\Sigma^{*}}^{-n})$ (right).}
	\label{projection1}
\end{figure}

\end{example}

\begin{example}\label{example2}
	
	We consider the function
	\begin{equation}\label{func2}
		f(\mathbf{x}) = \frac{1}{\sum_{i=1}^{d}x_i^r+c} \notag
	\end{equation}
	with $r\in\mathbb{N}$, $c\in \mathbb{R}$, and
	discuss three cases: $\mathcal{B}_1=\{r$ is odd, $|c|>d\}$, $\mathcal{B}_2=\{r$ is even, $c<-d\}$ and $\mathcal{B}_3=\{r$ is even, $c>0\}$.
	As shown in Appendix \ref{appendixA}, for the first two cases, we have
	\begin{equation}\label{Rtwo}
		R_\Sigma = \rho\left(\!\sqrt[r]{|c|/d}\right), \qquad
		R_{\Sigma^{*}} = \rho\left(\!\sqrt[r]{|c|-(d-1)}\right).
	\end{equation}
	For the last case, we still do not know the closed forms of $R_\Sigma$ and $R_{\Sigma^{*}}$. However, Bos and Levenberg in \cite[pp.~375-385]{bos2018bernstein} studied the Runge function $f(\mathbf{x})= (x_1^2 + \dots +x_d^2+c)^{-1}$ with $c>0$, which is a special example of $\mathcal{B}_3$ (i.e., $r=2$), and obtained
	\begin{equation}\label{bosformula}
		R_\Sigma = \rho\left( \textnormal{i} \sqrt{c/d} \right), \qquad R_{\mathbb{B}_{2}^{+}} = \rho\left( \textnormal{i} \sqrt{c}\right), \qquad
		R_{\Sigma^{*}} = \rho\left( \textnormal{i} \sqrt{c} \right).
	\end{equation}

	Figure \ref{projection2} illustrates the maximum errors of $S_{n\Sigma}(f)$ and $S_{n\Sigma^*}(f)$ for
	\begin{equation}
		f(x_1,x_2)= \frac{1}{x_1+x_2+2.15}, \qquad \frac{1}{x_1^4+x_2^4-2.2}, \qquad  \frac{1}{x_1^2+ x_2^2+0.5}, \notag
	\end{equation}
	which belong to $\mathcal{B}_1$, $\mathcal{B}_2$ and $\mathcal{B}_3$ respectively. We see that the numerical results are consistent with our theoretical results.

	\begin{figure}
		\centering
		\includegraphics[width=7.2cm,height=5.5cm]{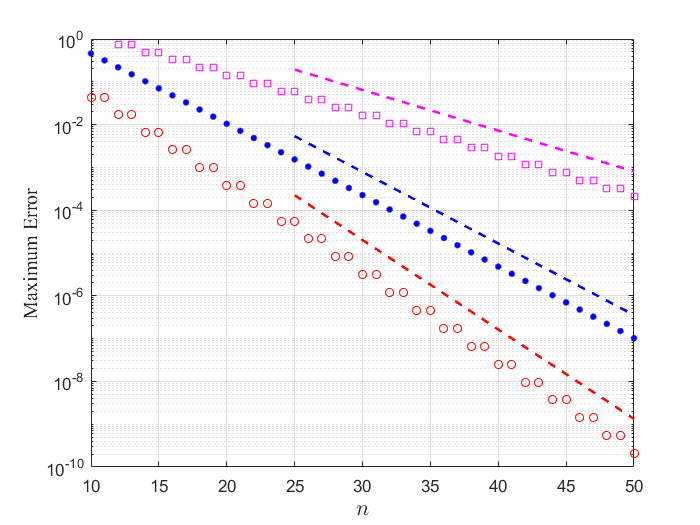}
		\includegraphics[width=7.2cm,height=5.5cm]{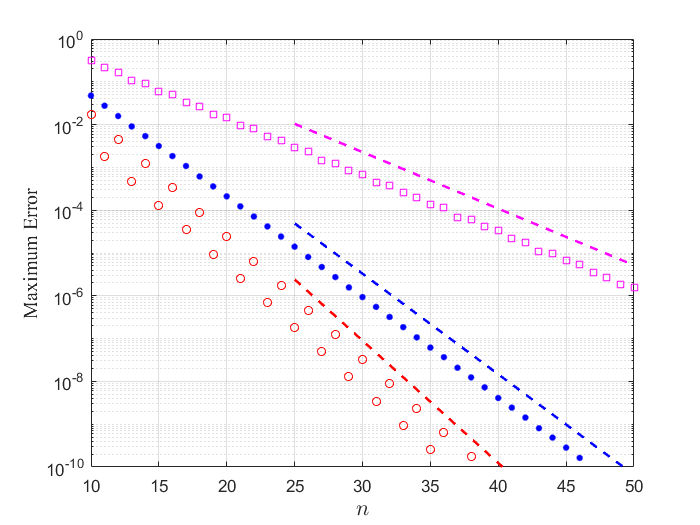}
		\caption{The maximum errors $\Vert f-S_{n\Sigma}(f) \Vert_{\infty}$ (left) and $\Vert f-S_{n\Sigma^*}(f) \Vert_{\infty}$ (right) for $f(x_1,x_2)=(x_1+x_2+2.15)^{-1}$ (dots), $(x_1^4+ x_2^4- 2.2)^{-1}$ (boxes) and $(x_1^2+ x_2^2+0.5)^{-1}$ (circles). The dashed lines denote the predicted rates $\mathcal{O}(R_\Sigma^{-n})$ (left) and $\mathcal{O}(R_{\Sigma^{*}}^{-n})$ (right).}
		\label{projection2}
	\end{figure}
	
\end{example}

\begin{remark}\label{rk:Trefethen}
Trefethen in \cite{trefethen2017multivariate} analyzed the exponential convergence rates of Chebyshev projections with total, Euclidean and maximal degrees. Specifically, under the assumption that $f$ is analytic in the $d$-dimensional region $\{\mathbf{z}\in\mathbb{C}^{d}:z_1^2+\cdots+z_d^2\in N_{d,h^2}\}$, where $h>0$ and $N_{d,h^2} \subset \mathbb{C}$ is the region bounded by the ellipse with foci $0$ and $d$ and leftmost point $-h^2$, Theorem 4.2 in  \cite{trefethen2017multivariate} shows that
	\begin{equation}
		\Vert f-S_{nP}(f) \Vert_{\infty}=	
		\begin{cases}
			\mathcal{O}\left(R^{-n/\sqrt{d}}\right), &P = \Sigma, \\[1ex]
			\mathcal{O}\left({R}^{-n}\right), &P = \mathbb{B}_{2}^{+} \textnormal{ or } P=\Sigma^{*},
		\end{cases} \notag
	\end{equation}
	where $R=h+\sqrt{h^2+1}$. Here we show that Trefethen's result for the total degree (i.e., $P = \Sigma$) is suboptimal. Consider the Runge function $f(\mathbf{x}) = (x_1^2 + \dots+x_d^2+h^2)^{-1}$ with $h>0$.
	While Trefethen's analysis predicts the convergence rate $\mathcal{O}(\widehat{R}^{-n})$ with $\widehat{R}=(h+\sqrt{h^2+1})^{1/\sqrt{d}}$, Theorem \ref{thm3} and (\ref{bosformula}) together yield the sharper rate $\mathcal{O}(R_\Sigma^{-n})$ with
	\begin{equation}
		R_{\Sigma}=\rho\left(\frac{\mathrm{i}h}{\sqrt{d}}\right)=\frac{h+\sqrt{h^2+d}}{\sqrt{d}}. \notag
	\end{equation}
	Note that $R_{\Sigma}>\widehat{R}$ for all $h>0$ and $d\geq2$. For instance, when $(d,h)=(2,1)$, we have $(\widehat{R},R_\Sigma) \approx (1.8649,1.9319)$, and when $(d,h)=(3,3)$, we have $(\widehat{R},R_\Sigma) \approx (2.8573,3.7321)$.
\end{remark}

\section{Applications}\label{section_application}
In this section, we apply our analysis results to several closely related topics, including tensorized Chebyshev interpolation, tensor product Gauss--Legendre quadrature, Padua interpolation and cubature, and Chebyshev-Galerkin method. We will show that the exponential convergence rates of these numerical algorithms are determined by the corresponding maximal convergence numbers.

\subsection{Tensorized Chebyshev interpolation}\label{section5.1}
In this subsection, we consider the error analysis of tensorized Chebyshev interpolation. For $\mathcal{N}=(N_1,\dots,N_d) \in \mathbb{N}_{0}^d$, let $I_{\mathcal{N}}$ be the polynomial interpolation operator corresponding to the tensor product Chebyshev grid $\mathcal{S}_1 \times \cdots \times \mathcal{S}_d$, where $\mathcal{S}_k = \{ \cos (j\pi /N_k) \}_{j=0}^{N_k}$.
It is known that
\begin{equation}
	I_{\mathcal{N}} = I_{x_1}^{N_1} \circ \cdots \circ I_{x_d}^{N_d}, \notag
\end{equation}
where $I_{x_k}^{N_k}$ denotes the interpolation operator for the variable $x_k$ at the nodes $\mathcal{S}_k$. Our error analysis relies on the following lemma, which provides a way to estimate the interpolation error from Chebyshev coefficients.
\begin{lemma}\label{aliasing}
	Assume that $f(\mathbf{x})$ satisfies the Dini--Lipschitz condition. Then for $\mathcal{N}=(N_1,\dots,N_d) \in \mathbb{N}_{0}^d$, it holds that
	\begin{equation}
		\Vert I_{\mathcal{N}}(f)-f \Vert_{\infty} \le
		2\left(\sum_{k_1=N_1+1}^{\infty}\sum_{k_2=0}^{\infty}  \cdots  \sum_{k_d=0}^{\infty} |a_{\mathbf{k}}|
		+ \cdots + \sum_{k_1=0}^{\infty}\sum_{k_2=0}^{\infty} \cdots \sum_{k_d=N_d+1}^{\infty} |a_{\mathbf{k}}| \right). \notag
	\end{equation}
\end{lemma}
\begin{proof}
We only prove the case for $d=2$, as the proof for $d\ge3$ is similar. We write the Chebyshev expansions of $f(x,y)$ and $I_x^{N_1}(f)(x,y)$ as
\begin{equation}
	\begin{split}
		&f(x,y) = \sum_{i=0}^{\infty}a_i(y)T_i(x) = \sum_{i=0}^{\infty}\sum_{j=0}^{\infty}a_{i,j}T_j(y)T_i(x), \\
		&I_x^{N_1}(f)(x,y) = \sum_{i=0}^{N_1}c_i(y)T_i(x) = \sum_{i=0}^{N_1}\sum_{j=0}^{\infty}c_{i,j}T_j(y)T_i(x) = \sum_{j=0}^{\infty} \left( \sum_{i=0}^{N_1}c_{i,j}T_i(x) \right)T_j(y). \notag
	\end{split}
\end{equation}
From \cite[Theorem 4.2]{trefethen2019approximation} we know that $\{ a_i(y) \}_{i=0}^\infty$ and $\{ c_i(y) \}_{i=0}^{N_1}$ satisfy the aliasing relationship
\begin{equation}
	\begin{split}
		&c_0(y) = \sum_{l=0}^{\infty} a_{2lN_1}(y), \quad \ c_{N_1}(y) = \sum_{l=0}^{\infty} a_{(2l+1)N_1}(y), \\
		&c_{i}(y) = a_i(y) + \sum_{l=1}^{\infty} \left( a_{2lN_1-i}(y) + a_{2lN_1+i}(y) \right),\quad 1\le i \le N_1-1, \notag
	\end{split}
\end{equation}
which implies that
\begin{equation}\label{cainq}
	\sum_{i=0}^{N_1}|c_{i,j}| \le \sum_{i=0}^{\infty}|a_{i,j}|, \quad \forall j \in \mathbb{N}_0. \nonumber
\end{equation}
Combining this with \cite[Equation (4.9)]{trefethen2019approximation}, we obtain
\begin{equation}
	\begin{split}
		\Vert I_y^{N_2}\circ I_x^{N_1}(f) - f \Vert_{\infty}
		&\le \Vert I_y^{N_2}\circ I_x^{N_1}(f) - I_x^{N_1}(f) \Vert_{\infty} +
		\Vert I_x^{N_1}(f) - f \Vert_{\infty} \\
		&\le 2\sum_{j=N_2+1}^{\infty} \max_{x\in [-1,1]} \left|\sum_{i=0}^{N_1}c_{i,j}T_i(x)\right| + 2\sum_{i=N_1+1}^{\infty} \max_{y\in [-1,1]} |a_{i}(y)| \\
		&\le 2\sum_{j=N_2+1}^{\infty}\sum_{i=0}^{N_1}|c_{i,j}| + 2\sum_{i=N_1+1}^{\infty} \max_{y\in [-1,1]} \left|\sum_{j=0}^{\infty}a_{i,j}T_j(y)\right| \\
		&\le 2\sum_{j=N_2+1}^{\infty}\sum_{i=0}^{\infty}|a_{i,j}| + 2\sum_{i=N_1+1}^{\infty}\sum_{j=0}^{\infty}|a_{i,j}|. \notag
	\end{split}
\end{equation}
This completes the proof.
\end{proof}
The following theorem establishes the exponential convergence rate of $I_{n\mathcal{N}}(f)$ with $n \in \mathbb{N}$ for multivariate analytic functions.
\begin{theorem}\label{thm5}
Let $\mathcal{N}=(N_1,\dots,N_d)$ be a $d$-dimensional positive integer vector, and define $P =[0,N_1]\times \cdots \times [0,N_d]$. Assume that $f(\mathbf{z})$ is analytic in a neighbourhood of $\overline{\Omega(P,R)}$ for some $R>1$. Then for $n \in \mathbb{N}$, it holds that
\begin{equation}\label{errInt}
	\Vert f-I_{n\mathcal{N}}(f) \Vert_{\infty} \le C_3(d,P,R,f) \left[   \sum_{k=1}^{d}\frac{(d-1)!}{(d-k)!} \frac{n^{d-k}}{(\log R)^k}\right] R^{-n}, \notag
\end{equation}
where
$C_3(d,P,R,f)=2d \, C_2(d,P,R,f)$ and $C_2(d,P,R,f)$ is defined in \textnormal{(\ref{const2})}.
\end{theorem}
\begin{proof}
From Lemma \ref{aliasing}, we know that
\begin{equation}
	\Vert f-I_{n\mathcal{N}}(f) \Vert_{\infty} \le 2d \sum_{\Vert \mathbf{k} \Vert_{P}>n} |a_{\mathbf{k}}|. \notag
\end{equation}
The desired error estimate follows by combining the above inequality and the proof of Theorem \ref{thm3}.
\end{proof}

Theorem \ref{thm5} shows that the convergence rate of Chebyshev interpolation $I_{n\mathcal{N}}(f)$ with the tensor product Chebyshev grid $\mathcal{S}_1 \times \cdots \times \mathcal{S}_d$ is the same as that of Chebyshev projection $S_{nP}(f)$ with $P=[0,N_1]\times \cdots \times [0,N_d]$. Both converge at the exponential rate $\mathcal{O}(R_P^{-n})$ with maximal convergence number
\begin{equation}\label{RN}
	R_P  = \inf _{ \mathbf{z} \in S(f)} \exp\left(\Vert \log \rho (\mathbf{z}) \Vert_{P^*} \right) = \inf _{ \mathbf{z} \in S(f)} \prod_{i=1}^{d} \rho(z_i)^{N_i}.
\end{equation}

The left subplot of Figure \ref{tencheb_and_tengauss} presents the maximum errors of $I_{n}(f):=I_{n\mathbf{1}}(f)$ for
\begin{equation}\label{testfunc_tencheb}
	f(x_1,x_2) = \frac{1}{x_1^2x_2^3+1.15}, \quad \frac{1}{x_1^2x_2^4-1.1},
	\quad \frac{1}{x_1+x_2+2.1}, \quad \frac{1}{x_1^4+x_2^4-2.05},
\end{equation}
which belong to $\mathcal{A}_1$, $\mathcal{A}_2$, $\mathcal{B}_1$ and $\mathcal{B}_2$, respectively.
Based on $R_{\Sigma^*}$ in (\ref{Rone}) and (\ref{Rtwo}), the corresponding $R_{\Sigma^*}$ for these functions are $\rho(1.15^{1/3})$, $\rho(1.1^{1/4})$, $\rho(1.1)$ and $\rho(1.05^{1/4})$, respectively.
We see that $\Vert f - I_{n}(f) \Vert_{\infty}$ decays exponentially like $\mathcal{O}(R_{\Sigma^{*}}^{-n})$
for each function, as predicted by Theorem \ref{thm5}.

\begin{remark}\label{contrast}
Under the same analyticity assumption as in Lemma \ref{lemmaB}, Sauter and Schwab established the exponential convergence rate $\mathcal{O}(\rho_{\min}^{-n})$ for the tensorized Chebyshev interpolation $I_n(f)$ in \cite[Lemma 7.3.3]{sauter2011boundary}, where $\rho_{\min}=\min_{1\le i\le d}\{\rho_i\}$. Building on this result, Ga\ss~ et al. derived a sharper bound in \cite[Proposition 2.1]{gass2018chebyshev}. However, the exponential convergence rate $\mathcal{O}(\rho_{\min}^{-n})$ has not seen substantial improvement. We claim that these results might be suboptimal for certain functions.
Taking $f(x_1,x_2)=(x_1+x_2+2.1)^{-1}$ in (\ref{testfunc_tencheb}) as an illustrative example, whose Bernstein polyellipse satisfies $(\rho_1+\rho_1^{-1})/2+(\rho_2+\rho_2^{-1})/2\le s$ for some $s\in(2,2.1)$, a simple calculation shows that the maximum of $\rho_{\min}$ is attained when $\rho_1=\rho_2=(s+\sqrt{s^2-4})/2<(21+\sqrt{41})/20$. However, as shown in Figure \ref{tencheb_and_tengauss}, the actual convergence rate is $\Vert f-I_n(f)\Vert_\infty = \mathcal{O}(R_{\Sigma^{*}}^{-n})$ with $R_{\Sigma^{*}} = \rho(1.1) = (11+\sqrt{21})/10 >(21+\sqrt{41})/20$.
\end{remark}

\begin{remark}
All the results and proof techniques in this subsection also apply to multivariate interpolations based on tensor product Chebyshev zeros. The aliasing formula corresponding to Chebyshev zeros can be found in \cite[p.~96]{boyd2001chebyshev}.
\end{remark}

\begin{figure}
\centering
\includegraphics[width=7.2cm,height=5.5cm]{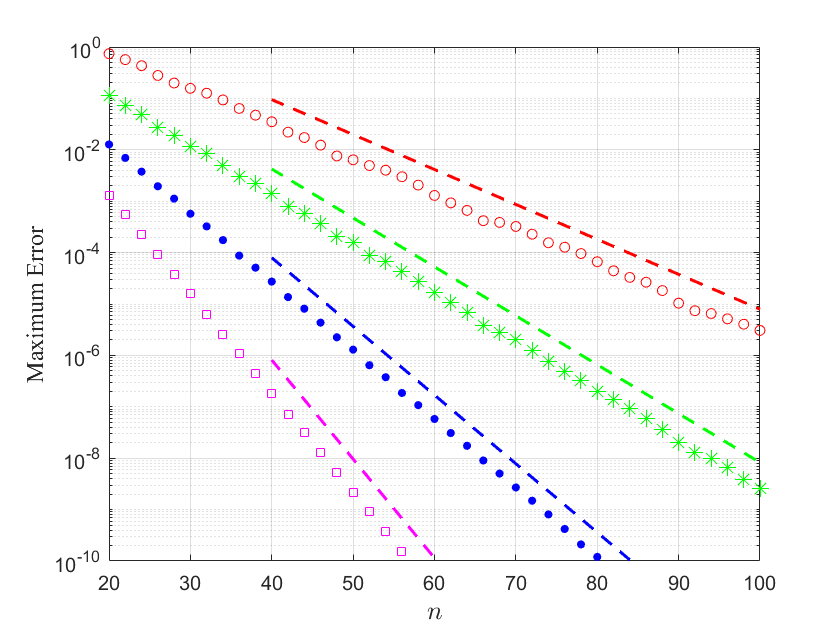}
\includegraphics[width=7.2cm,height=5.5cm]{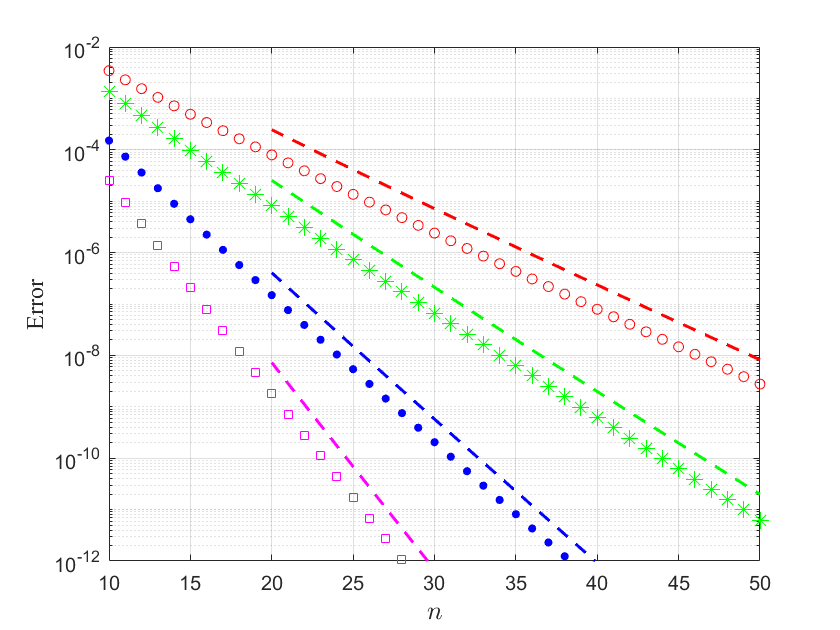}
\caption{The errors $\Vert f-I_{n}(f) \Vert_{\infty}$ (left) and $|Q(f)-Q_n(f)|$ (right) for $f(x_1,x_2)=(x_1^2x_2^3+1.15)^{-1}$ (dots), $(x_1^2x_2^4 - 1.1)^{-1}$ (asterisks), $(x_1+ x_2+2.1)^{-1}$ (boxes) and $(x_1^4+x_2^4-2.05)^{-1}$ (circles). The dashed lines denote the predicted rates $\mathcal{O}(R_{\Sigma^*}^{-n})$ (left) and $\mathcal{O}(R_{\Sigma^{*}}^{-2n})$ (right).}
\label{tencheb_and_tengauss}
\end{figure}

\subsection{Tensor product Gauss--Legendre quadrature}\label{section5.2}
For $\mathcal{N}=(N_1,\dots,N_d) \in \mathbb{N}^d$, let $Q_{\mathcal{N}}(f)$ denote the tensor product Gauss--Legendre quadrature with $N_i+1$ nodes for the $i$-th variable of $f$, i.e.,
\begin{equation}
	Q_\mathcal{N}(f) = \sum_{k_1=0}^{N_1}\cdots \sum_{k_d=0}^{N_d} \left( \prod_{j=1}^{d} w^{N_j}_{k_j} \right) f\big(x^{N_1}_{k_1},\dots,x^{N_d}_{k_d}\big), \notag
\end{equation}
where $\{x_k^N\}_{k=0}^N$ are the zeros of the Legendre polynomial of degree $N+1$, and $\{w_k^N\}_{k=0}^N$ are the corresponding Gauss--Legendre weights on $[-1,1]$.
Let $Q(f)$ denote the integral of $f$ over $[-1,1]^d$. The following theorem establishes the exponential convergence rate of $(Q-Q_{n\mathcal{N}})(f)$ with $n \in \mathbb{N}$ for multivariate analytic functions.

\begin{theorem}\label{thm6}
Let $\mathcal{N}=(N_1,\dots,N_d)$ be a $d$-dimensional positive integer vector, and define $P =[0,N_1]\times \cdots \times [0,N_d]$. Assume that $f(\mathbf{z})$ is analytic in a neighbourhood of $\overline{\Omega(P,R)}$ for some $R>1$. Then for $n \in \mathbb{N}$, it holds that
\begin{equation}
\left| Q(f)-Q_{n\mathcal{N}}(f) \right| \le C_4(d,P,R,f) \left[   \sum_{k=1}^{d}\frac{(d-1)!}{(d-k)!} \frac{(2n)^{d-k}}{(\log R)^k}\right] R^{-2n}, \notag
\end{equation}
where
$C_4(d,P,R,f)=2^{d+1}C_2(d,P,R,f)$ and $C_2(d,P,R,f)$ is defined in \textnormal{(\ref{const2})}.
\end{theorem}
\begin{proof}
From the fact that $(n+1)$-point Gauss--Legendre quadrature integrates univariate polynomials of degree $\le2n+1$ exactly, it follows that $Q(g)=Q_{n\mathcal{N}}(g)$ for all $g\in \mbox{Poly} \left( 2nP \right)$, which implies that
\begin{equation}
(Q-Q_{n\mathcal{N}})(f)= (Q-Q_{n\mathcal{N}})\left(f-S_{2nP}(f)\right). \notag
\end{equation}
Combining this with the identity $\sum_{k_{j}=0}^{N_{j}}w_{k_{j}}^{N_{j}}=2$ for all $1\le j\le d$, we obtain that
\begin{equation}
\begin{split}
	\left| (Q-Q_{n\mathcal{N}})(f)\right|
	&\le \left| Q\left(f-S_{2nP}(f)\right)\right|
	+\left|Q_{n\mathcal{N}}\left(f-S_{2nP}(f)\right)\right| \\
	&\le \left(\int_{[-1,1]^d}1\mathrm{d}\mathbf{x}+
	\sum_{k_1=0}^{N_1}\cdots \sum_{k_d=0}^{N_d}\prod_{j=1}^{d} w^{N_j}_{k_j}\right)\Vert f-S_{2nP}(f) \Vert_{\infty} \\
	&= 2^{d+1}\Vert f-S_{2nP}(f) \Vert_{\infty}. \notag
\end{split}
\end{equation}
The desired error estimate follows by using Theorem \ref{thm3} to estimate the last term of the above inequality.
\end{proof}

Theorem \ref{thm6} shows that the tensor product Gauss--Legendre quadrature $Q_{n\mathcal{N}}(f)$ converges at the exponential rate $\mathcal{O}(R_P^{-2n})$ for multivariate analytic functions, where the maximal convergence number $R_P$ is defined in (\ref{RN}). The right subplot of Figure \ref{tencheb_and_tengauss} illustrates the errors of $Q_{n}(f):=Q_{n\mathbf{1}}(f)$ for the test functions in (\ref{testfunc_tencheb}). We see that $\left| (Q-Q_n)(f) \right|$ decays exponentially like $\mathcal{O}(R_{\Sigma^{*}}^{-2n})$ for each function, as predicted by Theorem \ref{thm6}.

\subsection{Padua interpolation and cubature}\label{section5.3}
Let $\mathbb{P}^2_n$ denote the space of bivariate polynomials of total degree at most $n$, i.e., $\mathbb{P}^2_n = \mbox{Poly}(n\Sigma)$ with $\Sigma=\{ (x_1,x_2) \in(\mathbb{R}^+)^2: x_1 +x_2 \le 1 \}$.
For $n \in \mathbb{N}$, the set of the Padua points with cardinality $\textnormal{dim}(\mathbb{P}^2_n) = (n+1)(n+2)/2$ on $[-1,1]^2$ is defined as
\begin{equation}
\textnormal{Pad}_n := \left\{ \gamma_n\left(\frac{k \pi}{n(n+1)}\right), k = 0,\dots, n(n+1) \right\}, \notag
\end{equation}
where $\gamma_n(t) = (\cos(nt), \cos((n+1)t))$ traces the algebraic curve given by $T_{n+1}(x) = T_{n}(y)$ \cite{bos2006bivariate}. The Padua points are the first known example of optimal points for total degree polynomial interpolation in two variables, i.e., they are unisolvent and have a Lebesgue constant increasing like $\mathcal{O}(\log^2n)$ \cite[Section 3.2]{sommariva2008}.

Let $I_{\textnormal{Pad}_n}$ denote the interpolation operator at the Padua points $\textnormal{Pad}_n$, and let $Q_{\textnormal{Pad}_n}(f)$ denote the associated cubature formula, which is called the {\it nontensorial Clenshaw--Curtis cubature}, i.e.,
\begin{equation}
Q_{\textnormal{Pad}_n}(f) = \int_{[-1,1]^2} I_{\textnormal{Pad}_n}(f)(x_1,x_2) \textnormal{d}x_1 \textnormal{d}x_2. \notag
\end{equation}
It is known that $I_{\textnormal{Pad}_n}(f)$ and $Q_{\textnormal{Pad}_n}(f)$ are exact for $f \in \mathbb{P}^2_n$.
We refer to \cite{bos2006bivariate,caliari2008bivariate} for the derivation and properties of
Padua interpolation, and to \cite{sommariva2008} for those of nontensorial Clenshaw--Curtis cubature.
Let $Q(f)$ denote the integral of $f$ over $[-1,1]^2$.
The following theorem establishes the exponential convergence rate of $I_{\textnormal{Pad}_n}(f)$ and $Q_{\textnormal{Pad}_n}(f)$ for bivariate analytic functions.

\begin{theorem}\label{thm7}
Assume that $f(z_1,z_2)$ is analytic in a neighbourhood of $\overline{\Omega(\Sigma,R)}$ for some $R>1$. Then for $n \in \mathbb{N}$, it holds that
\begin{equation}
\Vert f-I_{\textnormal{Pad}_n}(f) \Vert_{\infty} \le C_5 \log^2(n) n R^{-n}, \qquad
\left| Q(f) - Q_{\textnormal{Pad}_n}(f)\right| \le C_6 n R^{-n},
\notag
\end{equation}
where $C_5$, $C_6$ are positive constants independent of $n$, but depending on $R$ and $f$.
\end{theorem}
\begin{proof}
For the first estimate, it follows from \cite[Theorem 3]{bos2006bivariate} that
\begin{equation}
\Vert f-I_{\textnormal{Pad}_n}(f) \Vert_{\infty}
\le C\log^2(n) \inf \left\{ \Vert f-p \Vert_{\infty}: p \in \mathbb{P}_n^2 \right\}
\le C\log^2(n) \Vert f-S_{n\Sigma}(f) \Vert_{\infty}, \notag
\end{equation}
where $C$ is a positive constant independent of $n$ and $f$. For the second estimate, it follows from
\cite[Equation (29)]{sommariva2008} that
\begin{equation}
\left| Q(f) - Q_{\textnormal{Pad}_n}(f)\right|
\le \pi^2 \inf \left\{ \Vert f-p \Vert_{\infty}: p \in \mathbb{P}_n^2 \right\}
\le \pi^2 \Vert f-S_{n\Sigma}(f) \Vert_{\infty}. \notag
\end{equation}
Using Theorem \ref{thm3} to estimate the last term of the above two inequalities gives the desired results.
\end{proof}

Theorem \ref{thm7} shows that the Padua interpolation and the nontensorial
Clenshaw--Curtis cubature converge at the exponential rate $\mathcal{O}(R_{\Sigma}^{-n})$ for bivariate analytic functions.
Figure \ref{padua_plot} illustrates the errors of $I_{\textnormal{Pad}_n}(f)$ and $Q_{\textnormal{Pad}_n}(f)$ for
\begin{equation}
f(x_1,x_2) = \frac{1}{x_1^2x_2^5+3.5}, \quad \frac{1}{x_1^3+ x_2^3-2.35},
\quad \frac{1}{x_1^2+x_2^2-2.5}, \quad \frac{1}{x_1^2+x_2^2+2}, \notag
\end{equation}
which belong to $\mathcal{A}_1$, $\mathcal{B}_1$, $\mathcal{B}_2$ and $\mathcal{B}_3$, respectively.
Based on $R_{\Sigma}$ in (\ref{Rone}), (\ref{Rtwo}) and (\ref{bosformula}), the corresponding $R_{\Sigma}$ for these functions are $\rho(3.5^{1/7})$, $\rho(1.175^{1/3})$, $\rho(1.25^{1/2})$ and $\rho(\textnormal{i})$, respectively.
We see that the numerical results are in agreement with our predictions for Padua interpolation. For Padua cubature, whose errors were computed using the Multiprecision Computing Toolbox, the convergence rates are slightly faster than the predicted rates $\mathcal{O}(R_{\Sigma}^{-n})$.

\begin{remark}
Numerical computations suggest that $|(Q-Q_{\textnormal{Pad}_n})(f)| = \mathcal{O}(n^{-3}R_{\Sigma}^{-n})$ for the above four test functions. This additional algebraic factor might be explained by the aliasing of Chebyshev polynomials on the Chebyshev extreme points \cite[Theorem 5.2]{trefethen2008gauss}. As our focus is the exponential convergence, we omit the discussion of this algebraic factor here.
\end{remark}

\begin{figure}
\centering
\includegraphics[width=7.2cm,height=5.5cm]{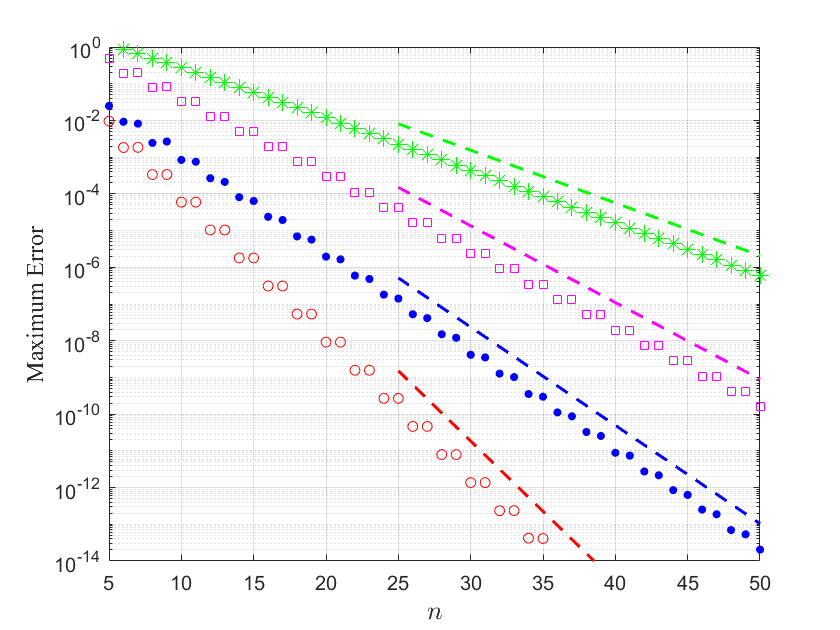}
\includegraphics[width=7.2cm,height=5.5cm]{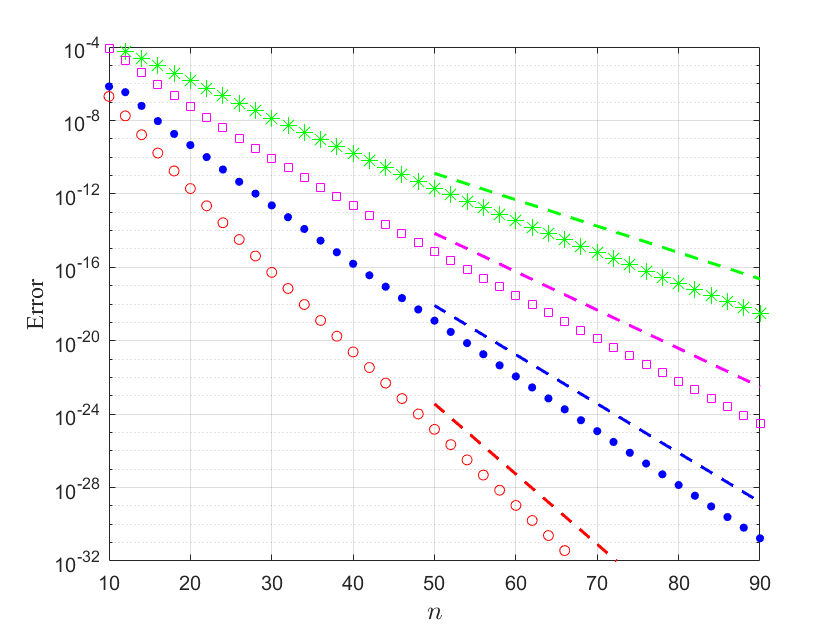}
\caption{The errors $\Vert f-I_{\textnormal{Pad}_n}(f) \Vert_{\infty}$ (left) and $|Q(f)-Q_{\textnormal{Pad}_n}(f)|$ (right) for $f(x_1,x_2)=(x_1^2x_2^5+3.5)^{-1}$ (dots), $(x_1^3+ x_2^3-2.35)^{-1}$ (asterisks), $(x_1^2+x_2^2-2.5)^{-1}$ (boxes) and $(x_1^2+x_2^2+2)^{-1}$ (circles). The dashed lines
denote the predicted rates $\mathcal{O}(R_{\Sigma}^{-n})$.}
\label{padua_plot}
\end{figure}

\subsection{Chebyshev-Galerkin method}\label{section5.4}
Consider the elliptic equation
\begin{equation}\label{elliptic}
\nu u(\mathbf{x}) -\Delta u(\mathbf{x}) = f(\mathbf{x}),\quad  \mathbf{x}\in \mathcal{D} :=(-1,1)^d,
\end{equation}
with homogeneous Dirichlet boundary conditions $u|_{\partial \mathcal{D}} = 0$ and $\nu \ge 0$.
Let $H_{0,\omega}^1(\mathcal{D}) = \{ u\in H_{\omega}^1(\mathcal{D}):u|_{\partial \mathcal{D}} = 0 \}$ be equipped with the norm $\Vert u\Vert_{1,\omega}=(\Vert u \Vert_{\omega}^2+\Vert \nabla u \Vert_{\omega}^2 )^{1/2}$, where $\Vert \cdot\Vert_\omega$ denotes the $L_{\omega}^2(\mathcal{D})$-norm with Chebyshev weight function $\omega(\mathbf{x})=\prod_{i=1}^d \left(1-x_i^2\right)^{-1/2}$. A weak formulation of (\ref{elliptic}) is to find $u\in H_{0,\omega}^1(\mathcal{D})$ such that
\begin{equation}\label{formulation1}
a(u,v):=\nu \int_{\mathcal{D}}uv\omega  + \int_{\mathcal{D}} \nabla u \cdot \nabla (v\omega) = \int_{\mathcal{D}}fv\omega := (f,v)_{\omega},\quad \forall v \in H_{0,\omega}^1(\mathcal{D}).
\end{equation}
From \cite[Lemma 8.7]{shen2011spectral} and the Lax--Milgram lemma \cite[Theorem B.1]{shen2011spectral}, it follows that the bilinear form $a(\cdot,\cdot)$ is continuous and coercive, and the solution of (\ref{formulation1}) exists and is unique.
Let $X_N^0 = \mathrm{span} \big\{ \! \prod_{i=1}^d \psi_{k_i}(x_i):0 \le k_i \le N-2, \, \psi_k(x) = T_k(x)-T_{k+2}(x) \big\}$.
The Chebyshev-Galerkin approximation for (\ref{formulation1}) is to find $u_N \in X_N^0$ such that
\begin{equation}\label{formulation2}
a(u_N,v_N) = (f,v_N)_{\omega}, \quad \forall v_N \in X_N^0.
\end{equation}
Using Theorem \ref{thm3} and some standard arguments, we derive the following convergence result for the numerical scheme (\ref{formulation2}).

\begin{theorem}\label{thmChebGalerkin}
Let $u$ and $u_N$ be the solutions of \eqref{formulation1} and \eqref{formulation2}, respectively. If $u$ is analytic in a neighbourhood of $\overline{\Omega(\Sigma^{*},R)}$ for some $R>1$, then for $N \in \mathbb{N}$, it holds that
\begin{equation}
\Vert u-u_N \Vert_{1,\omega} \le C_7 N^{d-1} R^{-N}, \notag
\end{equation}
where $C_7$ is a positive constant independent of $N$, but depending on $d,R,\nu,u$.
\end{theorem}
\begin{proof}
We only prove the case for $d=2$, as the proof for $d \ge 3$ is similar. We denote by $C$ a generic positive constant independent of $N$ that may have different values at different places.
From C\'{e}a's lemma \cite[Theorem 1.2]{shen2011spectral}, we know that $u_N$ exists and is unique, and satisfies
\begin{equation}\label{spec1}
\Vert u-u_N \Vert_{1,\omega} \le C \Vert u-v \Vert_{1,\omega},\quad  \forall v \in X_N^0.
\end{equation}
Noting that the partial derivatives of $u(x,y)$ of any order exist and are analytic in a neighbourhood of $\overline{\Omega(\Sigma^{*},R)}$,
we define
\begin{equation}
\begin{split}
\phi(x,y) = &\int_{-1}^y \int_{-1}^x p(\zeta,\eta)\mathrm{d}\zeta\mathrm{d}\eta
- \frac{x+1}{2}\int_{-1}^y \int_{-1}^1 p(\zeta,\eta)\mathrm{d}\zeta\mathrm{d}\eta
\\
&-\frac{y+1}{2}\int_{-1}^1 \int_{-1}^xp(\zeta,\eta)\mathrm{d}\zeta\mathrm{d}\eta + \frac{(x+1)(y+1)}{4}\int_{-1}^1 \int_{-1}^1 p(\zeta,\eta)\mathrm{d}\zeta\mathrm{d}\eta,
\end{split} \notag
\end{equation}
where $p(\zeta,\eta) = S_{(N-1)\Sigma^*}(u_{xy})(\zeta,\eta)$. A direct calculation gives $\phi \in X_N^0$ and
\begin{equation}\label{fourpart}
\begin{split}
\Vert u_x - \phi_x \Vert_{\omega} \le & \left\Vert u_x- \int_{-1}^y p(x,\eta)\mathrm{d}\eta \right\Vert_{\omega}
+ \frac{1}{2} \left\Vert\int_{-1}^y \int_{-1}^1 p(\zeta,\eta)\mathrm{d}\zeta \mathrm{d}\eta\right\Vert_{\omega}		 \\
& +\frac{1}{2}\left\Vert (y+1)\int_{-1}^1 p(x,\eta)\mathrm{d}\eta\right\Vert_{\omega}+ \frac{\Vert y+1 \Vert_{\omega}}{4}\left|\int_{-1}^1\int_{-1}^1 p(\zeta,\eta)\mathrm{d}\zeta\mathrm{d}\eta\right|. \notag
\end{split}
\end{equation}
Following similar steps to the proof of \cite[Theorem 3.38]{shen2011spectral},
we know that the four terms on the right side of the above inequality can be bounded by $\left\Vert u_{xy}-p \right\Vert_\infty$.
Applying a similar argument to $\Vert u_y-\phi_y \Vert_\omega$ and using the weighted Poincar\'{e} inequality \cite[Lemma 8.6]{shen2011spectral}, we have
\begin{equation}\label{spec2}
	\begin{split}
		\Vert u -\phi \Vert_{1,\omega}
		&\le C \left( \Vert u_x - \phi_x \Vert_{\omega}^2 + \Vert u_y - \phi_y \Vert_{\omega}^2 \right)^{1/2}	\\
		& \le C\Vert u_{xy}-p \Vert_\infty = C\Vert u_{xy}-S_{(N-1)\Sigma^*}(u_{xy}) \Vert_\infty.
	\end{split}
\end{equation}
Taking $v = \phi$ in \eqref{spec1} and using Theorem \ref{thm3} to estimate the last term of \eqref{spec2}, we obtain the desired result.
\end{proof}

To validate our theoretical results, we employ the Chebyshev-Galerkin method proposed in \cite{shen1995efficient} to solve the equation (\ref{elliptic}) for $d=2$ and $\nu = 1$, taking $\widetilde{u}(x,y)$ as the exact solution. Here, $\widetilde{u}(x,y)$ denotes $u(x,y)$ with the boundary data lifted according to \cite[Section 4.2]{shen1994efficient}, i.e., $\widetilde{u}|_{\partial\mathcal{D}} = 0$.
It is easy to verify that the maximal convergence number $R_{\Sigma^*}$ for $\widetilde{u}(x,y)$ is equal to  $R_{\Sigma^*}$ for $u(x,y)$.
Figure \ref{ChebGalerkin} illustrates the $H_{\omega}^1(\mathcal{D})$-errors $\Vert \widetilde{u}-\widetilde{u}_N \Vert_{1,\omega}$ for
\begin{equation}
u(x,y) = \frac{\sin(x+y)}{2+x^3 y^4},\quad \frac{4x^2+3y^3}{2.5+x+y},\quad
\frac{e^{\sin(x+1.2y)}}{2.35+x^3+y^3},\quad \frac{e^{0.3x+\cos{y}}}{0.09+x^2+y^2}. \notag
\end{equation}
It follows from (\ref{Rone}), (\ref{Rtwo}) and (\ref{bosformula}) that the values of $R_{\Sigma^*}$ for each $u(x,y)$ are $\rho(2^{1/4})$, $\rho(1.5)$, $\rho(1.35^{1/3})$ and $\rho(0.3\mathrm{i})$, respectively. We see that $\Vert \widetilde{u}-\widetilde{u}_N \Vert_{1,\omega}$ decays exponentially like $\mathcal{O}(R_{\Sigma^*}^{-N})$ for each case, consistent with our theoretical results.

\begin{figure}
\centering
\includegraphics[width=8.4cm,height=6.2cm]{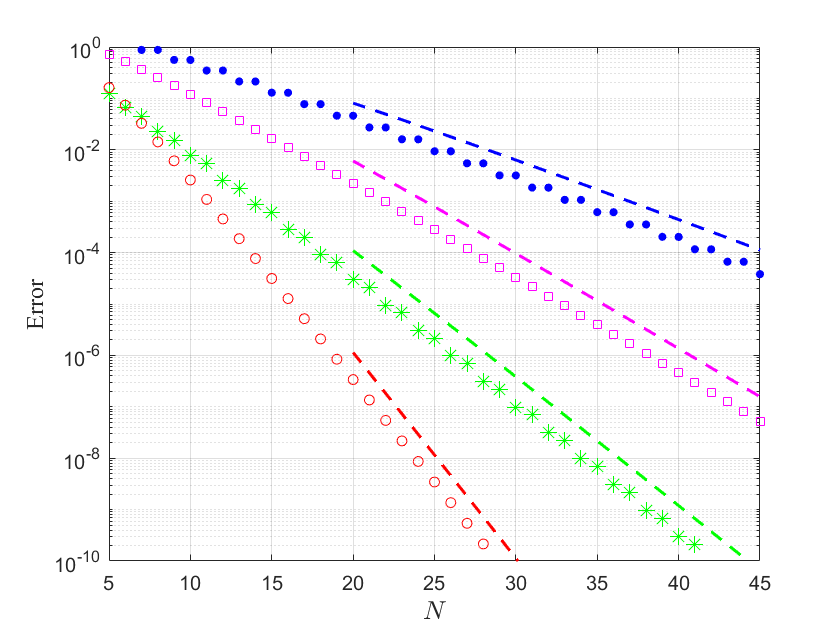}
\caption{The $H_{\omega}^1(\mathcal{D})$-errors $\Vert \widetilde{u}-\widetilde{u}_N \Vert_{1,\omega}$ for $u(x,y) = \sin(x+y)/(2+x^3 y^4)$ (asterisks), $(4x^2+3y^3)/(2.5+x+y)$ (circles), $e^{\sin(x+1.2y)}/(2.35+x^3+y^3)$ (boxes) and $e^{0.3x+\cos{y}}/(0.09+x^2+y^2)$ (dots).
The dashed lines denote the predicted convergence rates $\mathcal{O}(R_{\Sigma^*}^{-N})$.}
\label{ChebGalerkin}
\end{figure}

\appendix
\section{Proofs of (\ref{Rone}), (\ref{Rspec}) and (\ref{Rtwo})}\label{appendixA}
The idea for the proofs of (\ref{Rone}), (\ref{Rspec}) and (\ref{Rtwo}) is to first obtain a lower bound for $R_\Sigma$ or $R_{\Sigma^{*}}$, and then prove that this lower bound can be attained. The following results about $\rho(z)$ are required.

\begin{lemma}\label{lemmaA}
For the function $\rho(z)=|z+\sqrt{z^2-1}|$, where the branch is chosen so that $\rho(z) \ge 1$, the following statements are true.

\begin{itemize}
\setlength{\itemsep}{0pt}
\item[\rm (i)]
For $z\in \mathbb{C}$, it holds that $\rho(-z)=\rho(z)$, $\rho(\bar{z})=\rho(z)$.

\item[\rm (ii)]
For $z\in \mathbb{C}$, it holds that $\rho(|z|)\le \rho(z)$. The equality holds if and only if $z\in \mathbb{R}$.

\item[\rm (iii)]
For $x_i\ge 1,\ 1\le i \le d$, it holds that
\begin{equation}
\rho\left( \prod_{i=1}^{d}x_i \right)\le \prod_{i=1}^{d}\rho (x_i). \notag
\end{equation}
The equality holds if and only if $d-1$ terms in $\{ x_i \}_{i=1}^d$ are equal to one.
\end{itemize}
\end{lemma}

\begin{proof}
We note that the geometric meaning of $\rho(z)$ is the sum of the semi-major axis and the semi-minor axis of the Bernstein ellipse passing through the point $z$.
The first statement follows from the symmetry of the Bernstein ellipse about the real axis and imaginary axis. The second statement is trivial for $|z|\le1$, while for $|z|>1$, it follows from the fact that all points on the Bernstein ellipse $\partial E(\tilde{\rho})$ with $\tilde{\rho}=\rho(|z|)$ except $\pm |z|$ are located in the open disk $\{ t\in \mathbb{C}:|t|<|z|\}$.

For the third statement, it suffices to prove that $\rho(x_1 x_2) \le \rho(x_1)\rho(x_2)$ for $x_1, x_2\ge1$. A direct calculation gives
\begin{equation}
	\begin{split}
		&\rho(x_1)\rho(x_2)-\rho(x_1x_2)\ge 0 \\ \Leftrightarrow  &\left[x_1+(x_1^2-1)^{1/2}\right]\left[x_2+(x_2^2-1)^{1/2}\right]-\left[x_1x_2+(x_1^2x_2^2-1)^{1/2}\right] \ge 0 \\
		\Leftrightarrow
		&\left[x_1(x_2^2-1)^{1/2} + x_2(x_1^2-1)^{1/2} + (x_1^2-1)^{1/2}(x_2^2-1)^{1/2}\right]^2 \ge x_1^2x_2^2-1 \\ \Leftrightarrow
		&\left[x_1+(x_1^2-1)^{1/2}\right]\left[x_2+(x_2^2-1)^{1/2}\right] (x_1^2-1)^{1/2}(x_2^2-1)^{1/2} \ge 0.
		\notag
	\end{split}
\end{equation}
The last inequality clearly holds, which ends the proof.
\end{proof}

\begin{proof}[Proof of \rm{\eqref{Rone}}.]
We only prove the case for $f \in \mathcal{A}_1$,
as the proof for $f \in \mathcal{A}_2$ is similar. From the fact that $S(f) = \{ \mathbf{z} \in \mathbb{C}^d: \prod_{i=1}^{d}z_i^{r_i} = -c \}$, we have
\begin{equation}\label{mainA}
\prod_{i=1}^{d}|z_i|^{r_i} = |c|,\quad \forall(z_1,\dots,z_d)\in S(f).
\end{equation}
Without loss of generality, we assume that $r_d$ is odd.

We first derive the closed form of $R_{\Sigma}$. It follows from \eqref{mainA} that, for any $(z_1,\dots,z_d)\in S(f)$, there exists $z_j$ satisfying $|z_j| \ge |c|^{1/ \Vert \mathbf{r} \Vert_{\ell^1}}$. Hence
\begin{equation}
\max_{1\leq i\leq d} \{ \rho(z_i) \} \ge \rho (z_j) \ge \rho(|z_j|) \ge
\rho \left(|c|^{1/ \Vert \mathbf{r} \Vert_{\ell^1}} \right), \quad \forall(z_1,\dots,z_d)\in S(f), \notag
\end{equation}
where the second inequality follows from the item (ii) of Lemma \ref{lemmaA}. On the other hand, let
\[
y_1=\dots=y_{d-1}=|c|^{1/ \Vert \mathbf{r} \Vert_{\ell^1}}, \quad y_d=-\left(c/|c|\right)|c|^{1/ \Vert \mathbf{r} \Vert_{\ell^1}}.
\]
Then $(y_1,\dots,y_d)\in S(f)$ and $\max_{1\leq i\leq d} \{ \rho(y_i) \}=\rho(|c|^{1/ \Vert \mathbf{r} \Vert_{\ell^1}})$. Combining the above argument gives $R_\Sigma = \rho( |c|^{1/\Vert \mathbf{r} \Vert_{\ell^1}})$.

We next derive the closed form of $R_{\Sigma^*}$. Let $(z_1,\dots,z_d)\in S(f)$ and select components with the modulus greater than or equal to one, which we may assume to be $\{z_i\}_{i=1}^m \,(m\le d)$.
From the item (ii) and (iii) of Lemma \ref{lemmaA}, we know that
\begin{equation}\label{ri11}
\prod_{i=1}^d\rho(z_i) \ge \prod _{i=1}^d\rho(|z_i|) = \prod _{i=1}^m\rho(|z_i|)
\ge \rho \left(\prod_{i=1}^m|z_i|\right).
\end{equation}
Noting that $\prod _{i=1}^m|z_i|^{\Vert \mathbf{r} \Vert_{\ell^\infty}} \ge |c|$, which follows from (\ref{mainA}), we have
\begin{equation}\label{ri12}
\rho \left(\prod_{i=1}^m|z_i|\right) = \rho\left[ \left(\prod _{i=1}^m|z_i|^{\Vert \mathbf{r} \Vert_{\ell^\infty}}\right)^{1/{\Vert \mathbf{r} \Vert_{\ell^\infty}}}\right] \ge \rho\left(|c|^{1/{\Vert \mathbf{r} \Vert_{\ell^\infty}}}\right).
\end{equation}
Combining \eqref{ri11} and \eqref{ri12} gives $R_{\Sigma^{*}} \ge \rho( |c|^{1/\Vert \mathbf{r} \Vert_{\ell^\infty}})$.
On the other hand, selecting $r_j\in \{ r_i\}_{i=1}^{d}$ satisfying $r_j=\Vert \mathbf{r} \Vert_{\ell^\infty}$ (or either one if $r_j$ is not unique), we set
\begin{equation}
(y_1,\dots,y_d)=
\begin{cases}
y_d=-\left(c/|c|\right)|c|^{1/\Vert \mathbf{r} \Vert_{\ell^\infty}},\
y_i=1 \ (i\neq d), &r_j=r_d, \\
y_j=|c|^{1/\Vert \mathbf{r} \Vert_{\ell^\infty}},\ y_d=-c/|c|,\ y_i=1 \ (i\neq j,d), &r_j\neq r_d.
\end{cases} \notag
\end{equation}
Then $(y_1,\dots,y_d)\in S(f)$ and $\prod_{i=1}^d \rho(y_i)=\rho(|c|^{1/ \Vert \mathbf{r} \Vert_{\ell^\infty}})$. Combining the above argument gives $R_{\Sigma^{*}} = \rho(|c|^{1/\Vert \mathbf{r} \Vert_{\ell^\infty}})$.
\end{proof}

\begin{proof}[Proof of \rm{\eqref{Rspec}}.]
Let $x_j = \mathrm{Re}(z_j)$ and $y_j = \mathrm{Im}(z_j)$, $j = 1,2$. From the fact that $S(f) = \{ (z_1, z_2) \in \mathbb{C}^2: z_1 z_2 = \pm \mathrm{i} \sqrt{c} \}$, we have
\begin{equation}\label{mainK}
	\left| x_1 y_2 + x_2y_1 \right| = \sqrt{c},\quad \forall(z_1,z_2) \in S(f).
\end{equation}

We first derive the closed form of $R_{\Sigma}$.
Let $(z_1,z_2)\in S(f)$ and set $r := r(z_1,z_2)=\max\{\rho(z_1),\rho(z_2)\}$.
From the fact that $\rho(z_1), \rho(z_2) \le r$ and $\rho(z) \le R \Leftrightarrow z \in \overline{E(R)}$, we have
\[
\frac{x_j^2}{(r + r^{-1})^2}
+
\frac{y_j^2}{(r - r^{-1})^2}
\le \frac{1}{4},
\quad j=1,2.
\]
Using the above inequalities, the Cauchy--Schwarz inequality, and \eqref{mainK}, we obtain
\[
\begin{aligned}
	\frac{r^2 - r^{-2}}{4}
	&\ge \left( r^2 - r^{-2} \right) \left[ \frac{x_1^2}{(r+r^{-1})^2} + \frac{y_1^2}{(r-r^{-1})^2} \right]^{1/2}
	\left[ \frac{x_2^2}{(r+r^{-1})^2} + \frac{y_2^2}{(r-r^{-1})^2} \right]^{1/2}
	\\
	&\ge \left( r^2 - r^{-2} \right) \left| \frac{x_1}{r + r^{-1}} \frac{y_2}{r-r^{-1}} +  \frac{x_2}{r+r^{-1}} \frac{y_1}{r-r^{-1}} \right|
	= \left| x_1 y_2 + x_2 y_1 \right| = \sqrt{c}.
\end{aligned}
\]
Since $g(x) = x^2 - x^{-2}$ is strictly increasing on $[1, \infty)$, the above inequality implies that
\[
\max\{\rho(z_1),\rho(z_2)\} = r \ge \sqrt{\sqrt{1+ 4c} + 2\sqrt{c} } = \rho \left(\sqrt{s+1} \right), \quad \forall (z_1,z_2) \in S(f),
\]
where $s = (\sqrt{1+4c} - 1)/2$. The above lower bound can be attained, for example, at
$(\sqrt{s+1},\mathrm{i}\sqrt{s}) \in S(f)$. Combining the above argument gives $R_\Sigma = \rho (\sqrt{s+1} )$.

We next derive the closed form for $R_{\Sigma^*}$. Let $(z_1,z_2)\in S(f)$ and set $u_j := u_j(z_j) = \log \rho(z_j)$, $j = 1,2$.
From the fact that $z_j \in \partial E(e^{u_j})$, $j=1,2$, we have
\[
\frac{x_j^2}{\cosh^2u_j} + \frac{y_j^2}{\sinh^2u_j} = 1, \quad j=1,2.
\]
Let $m(u_1,u_2) = \max\{ \cosh u_1 \sinh u_2, \, \sinh u_1 \cosh u_2 \}$. Using the above equalities, the Cauchy--Schwarz inequality, and \eqref{mainK}, we obtain
\[
\begin{aligned}
	\sinh(u_1 + u_2)
	&= \cosh u_1 \sinh u_2 + \sinh u_1 \cosh u_2 \ge m(u_1,u_2) \\
	&= m(u_1,u_2) \left( \frac{x_1^2}{\cosh^2u_1} + \frac{y_1^2}{\sinh^2u_1} \right)^{1/2}
	\left( \frac{x_2^2}{\cosh^2u_2} + \frac{y_2^2}{\sinh^2u_2} \right)^{1/2} \\
	&\ge m(u_1,u_2) \left(\left| \frac{x_1}{\cosh u_1} \frac{y_2}{\sinh u_2}\right| + \left|\frac{x_2}{\cosh u_2} \frac{y_1}{\sinh u_1} \right| \right)
	\ge \left| x_1 y_2 + x_2 y_1 \right| = \sqrt{c}.
\end{aligned}
\]
Since $g(x) = \sinh(x)$ is strictly increasing on $[0, \infty)$, the above inequality implies that
\[
\rho(z_1)\rho(z_2) = e^{u_1+u_2} \ge e^{\operatorname{arsinh}\sqrt{c}} = \rho\left(\mathrm{i}\sqrt{c}\right), \quad \forall (z_1,z_2) \in S(f).
\]
The above lower bound can be attained, for example, at $(1, \mathrm{i}\sqrt{c}) \in S(f)$. Combining the above argument gives $R_{\Sigma^*} = \rho (\mathrm{i}\sqrt{c})$.
\end{proof}

\begin{proof}[Proof of \rm{\eqref{Rtwo}}.]
We only prove the case for $f \in \mathcal{B}_1$,
as the proof for $f \in \mathcal{B}_2$ is similar. From the fact that $S(f) = \{ \mathbf{z} \in \mathbb{C}^d: \sum_{i=1}^{d}z_i^{r} = -c \}$, we have
\begin{equation}\label{mainB}
\sum_{i=1}^{d}|z_i|^{r} \ge |c|,\quad \forall(z_1,\dots,z_d)\in S(f).
\end{equation}

We first derive the closed form of $R_{\Sigma}$. It follows from \eqref{mainB} that, for any $(z_1,\dots,z_d)\in S(f)$, there exists $z_j$ satisfying $|z_j| \ge \sqrt[r]{|c|/d}$. The remaining steps are similar to the part about $R_\Sigma$ in the proof of (\ref{Rone}).

We next derive the closed form of $R_{\Sigma^*}$. Let $(z_1,\dots,z_d)\in S(f)$ and select components with the modulus greater than or equal to one, which we may assume to be $\{z_i\}_{i=1}^m \,(m\le d)$. Following similar steps to the part about $R_{\Sigma^{*}}$ in the proof of (\ref{Rone}), we have
\begin{equation}
\prod _{i=1}^d\rho(z_i)
\ge \rho \left(\prod_{i=1}^m|z_i|\right)
= \rho\left[ \left(\prod _{i=1}^m|z_i|^{r}\right)^{1/r}\right]
\ge \rho\left[ \left(\sum _{i=1}^m|z_i|^{r} -(m-1)\right)^{1/r}\right], \notag
\end{equation}
where we have used the inequality $\prod_{i=1}^m \alpha_i \ge \sum_{i=1}^{m} \alpha_i -(m-1)$ for $\alpha_i\ge1$ at the second inequality. Noting that $\sum _{i=1}^m|z_i|^{r}+d-m \ge |c|$, which follows from (\ref{mainB}), we obtain
\begin{equation}
\rho\left[ \left(\sum _{i=1}^m|z_i|^{r} -(m-1)\right)^{1/r}\right] \ge
\rho\left(\! \sqrt[r]{|c|-(d-1)}\right). \notag
\end{equation}
To show that the above lower bound can be attained, we set
\begin{equation}
(y_1,\dots,y_d)=
\begin{cases}
y_d=-\sqrt[r]{c-(d-1)},\ y_i=-1\ (i\neq d), & c>d, \\
y_d=\sqrt[r]{-c-(d-1)},\ y_i=1\ (i\neq d), & c<-d.
\end{cases} \notag
\end{equation}
Then $(y_1,\dots,y_d)\in S(f)$ and $\prod_{i=1}^d \rho(y_i)=\rho\big(\!\sqrt[r]{|c|-(d-1)}\big)$. Combining the above argument gives $R_{\Sigma^{*}} = \rho\big(\!\sqrt[r]{|c|-(d-1)}\big)$.
\end{proof}

\section*{Acknowledgement}
This work was partially supported by the National Natural Science Foundation of China under grant number 12371367 and the Hubei Provincial Natural Science Foundation of China under grant number 2023AFA083 and the fundamental research funds for the central universities under grant number 2025BRSXA003.

\end{document}